\documentclass[preprint,3p,10pt]{elsarticle}

\usepackage{amssymb,amsthm}
\usepackage{latexsym,amsmath,subfigure,color}

\usepackage{lineno}

\journal{}

\newcommand{\eps}{\varepsilon}
\newcommand{\set}[1]{\left\{#1\right\}}

\newcommand{\p}{\partial}
\newcommand{\mB}{\mathbf{B}}
\newcommand{\mE}{\mathbf{E}}
\newcommand{\mG}{\mathbf{G}}
\newcommand{\mH}{\mathbf{H}}
\newcommand{\mU}{\mathbf{U}}
\newcommand{\mV}{\mathbf{V}}
\newcommand{\mW}{\mathbf{W}}
\newcommand{\md}{\mathbf{d}}
\newcommand{\mr}{\mathbf{r}}

\newcommand{\vt}{\boldsymbol{\theta}}

\theoremstyle{plain}
\newtheorem{thm}{Theorem}[section]

\theoremstyle{remark}
\newtheorem{rem}{Remark}[section]

\begin{document}

\begin{frontmatter}



\title{Real-time microwave imaging of unknown anomalies via scattering matrix}

\author{Won-Kwang Park}
\ead{parkwk@kookmin.ac.kr}
\address{Department of Information Security, Cryptography, and Mathematics, Kookmin University, Seoul, 02707, Korea.}

\begin{abstract}
  We consider an inverse scattering problem to identify the locations or shapes of unknown anomalies from scattering parameter data collected by a small number of dipole antennas. Most of researches does not considered the influence of dipole antennas but in the experimental simulation, they are significantly affect to the identification of anomalies. Moreover, opposite to the theoretical results, it is impossible to handle scattering parameter data when the locations of the transducer and receiver are the same in real-world application. Motivated by this, we design an imaging function with and without diagonal elements of the so-called scattering matrix. This concept is based on the Born approximation and the physical interpretation of the measurement data when the locations of the transducer and receiver are the same and different. We carefully explore the mathematical structures of traditional and proposed imaging functions by finding relationships with the infinite series of Bessel functions of integer order. The explored structures reveal certain properties of imaging functions and show why the proposed method is better than the traditional approach. We present the experimental results for small and extended anomalies using synthetic and real data at several angular frequencies to demonstrate the effectiveness of our technique.
\end{abstract}

\begin{keyword}
Microwave imaging \sep scattering matrix \sep scattering parameter data \sep Bessel functions \sep experimental results


\end{keyword}

\end{frontmatter}




\section{Introduction}
Generally, the purpose of the inverse scattering problem is to identify characteristics of unknown defects that cannot be observed directly, such as size, location, shape, and electric and magnetic properties, based on measured scattered-field or scattering parameter (or $S-$parameter) data. This is an old problem, known to be difficult due to the intrinsic ill-posedness and nonlinearity. Nevertheless, it remains an interesting problem in current science and mathematics because it has a wide range of applications, such as in breast-cancer detection \cite{HSM2,MFLPP,SKVH}, brain-stroke diagnosis \cite{MAMI,PFTYMPKE,SVMM}, ground-penetrating radar for mine detection \cite{GCGGC,HSR,TTC}, and in damage detection \cite{CPL,FFK,KJFF,VS}, which are highly related to safety and reliability issues in human life.

To solve this interesting problem, various inversion techniques and related computational methods have been investigated; examples include the Newton or Gauss-Newton method \cite{AS3,JPH,RMMP,SBSSNST}, the Levenberg-Marquadt algorithm \cite{CJGT,CM2,KSY}, the level-set technique \cite{AKM,DL,S1}, and the optimal control approach \cite{AGJKLY,AGKLS,HSZ3}, most of which are based on an iteration scheme; these techniques help in obtaining the characteristics of the target (minimizer), which minimizes the discrete norm (usually, $L^2-$norm) between the measured data in the presence of true and artificial targets at each iteration procedure. Although these techniques have been proven to be feasible in determining the characteristics of unknown targets, some conditions such as \textit{a priori} information, choice of an appropriate regularization term, and evaluation of the complex Fr{\'e}chet (or domain) derivative at each iteration step must be fulfilled to guarantee a successful procedure. Nevertheless, if one begins the iteration procedure with a bad initial guess far from the unknown target, one faces non-convergence, the local minimizer problem, and the requirement of a large computational cost. Hence, it is natural to investigate a fast algorithm for obtaining at least a good initial guess without any \textit{a priori} information about the targets.

Recently, various non-iterative techniques have been investigated, including MUltiple SIgnal classification (MUSIC) \cite{AKKLV,K1,P-MUSIC1}, direct-sampling method \cite{CILZ,IJZ1,IJZ2}, linear-sampling method \cite{CC,CGK,KR}, and topological derivatives \cite{GP,LR,P-TD3}. Subspace and Kirchhoff migrations are also known as non-iterative techniques in inverse scattering problem. When total number of directions of the incident field and corresponding scattered fields is sufficiently large, it has been confirmed that the Kirchhoff and subspace migrations operated at single and multiple time-harmonic frequencies are effective, stable, and robust non-iterative techniques. Refer to \cite{AGKPS,AGS1,BPV,HHSZ,P-SUB3} and the references therein. Following the traditional results, subspace migration based imaging algorithms were established for cases where every far-field element of the so-called scattering matrix is collectable; the total number of transducers and receivers (here, dipole antennas) are sufficiently large, and measurement data is not affected by the dipole antennas. However, in real-world application, when the locations of the transducer and receiver are the same, measurement data (diagonal elements of the scattering matrix) are influenced by not only anomalies but also by antennas (see Figure \ref{AntennaEffects}). Furthermore, sometimes manufacturing microwave systems that can measure scattered field data with the transducer and receiver at the same location is inconvenient. Hence, considering the mentioned situation, designing an alternative imaging algorithm is an interesting direction for research.

\begin{figure}[h]
\begin{center}
\subfigure[$S_{\mathrm{tot}}(n,n)$]{\includegraphics[width=0.16\textwidth]{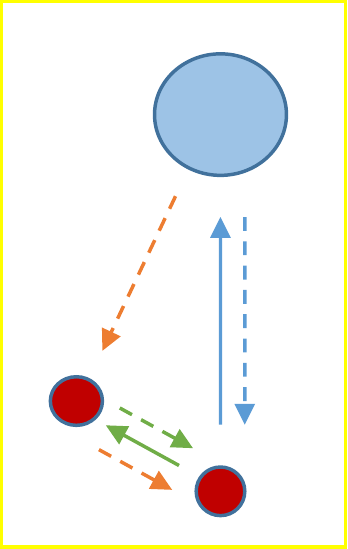}}
\subfigure[$S_{\mathrm{inc}}(n,n)$]{\includegraphics[width=0.16\textwidth]{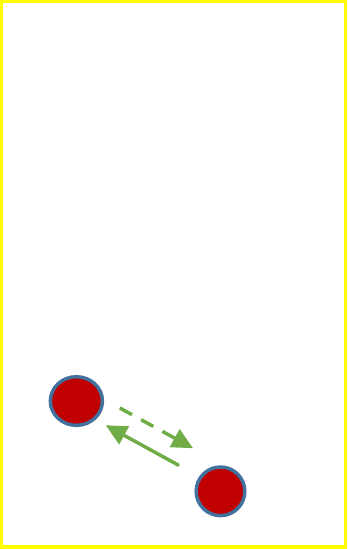}}
\subfigure[$S_{\mathrm{scat}}(n,n)$]{\includegraphics[width=0.16\textwidth]{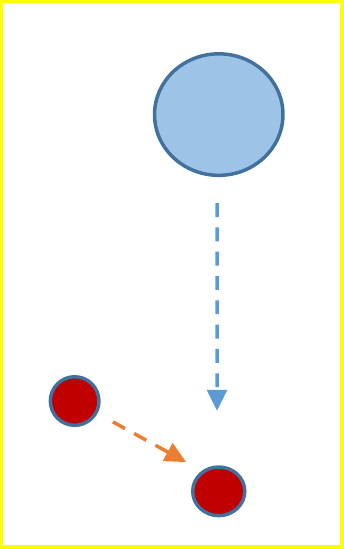}}
\subfigure[$S_{\mathrm{tot}}(m,n)$]{\includegraphics[width=0.16\textwidth]{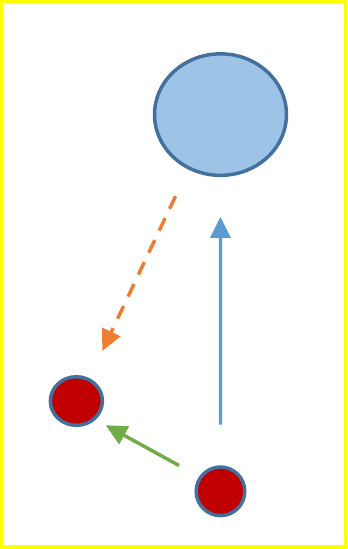}}
\subfigure[$S_{\mathrm{inc}}(m,n)$]{\includegraphics[width=0.16\textwidth]{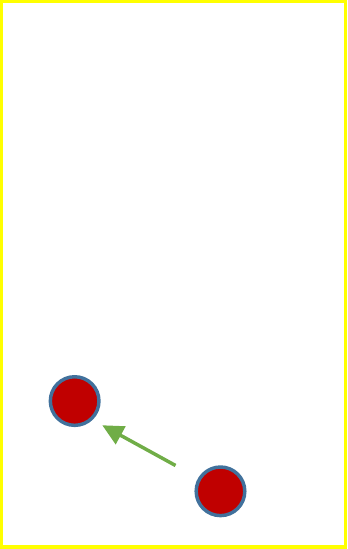}}
\subfigure[$S_{\mathrm{scat}}(m,n)$]{\includegraphics[width=0.16\textwidth]{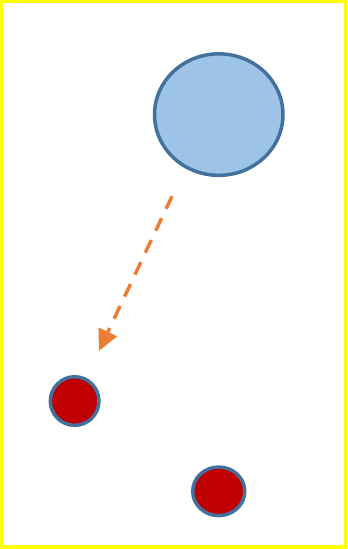}}
\caption{\label{AntennaEffects}Illustration of the influence of antenna when the locations of the transducer and receiver are the same (a)--(c) and different (d)--(f). Red-colored circles are dipole antennas $\md_n$ and $\md_m$, navy-blue colored circle is anomaly $\Sigma$. Solid and dashed arrows describe the incident and corresponding reflected fields, respectively. Definition of $S_{\mathrm{tot}}(m,n)$, $S_{\mathrm{inc}}(m,n)$, $S_{\mathrm{scat}}(m,n)$, $\md_n$, and $\Sigma$ are given in Section \ref{sec:2}.}
\end{center}
\end{figure}

Recently, the MUSIC algorithm and direct sampling method for identifying the location of the dielectric anomaly from scattering parameters collected by a small number of dipole antennas has been developed when the diagonal elements of the scattering matrix or scattering parameter with the same transducer and receiver location are measurable \cite{PKLS,P-DSM2}. In \cite{P-SUB9}, subspace migration for imaging of thin inhomogeneity without diagonal elements of so-called multi-static response (MSR) matrix whose elements are far-field pattern has been concerned. However, to the best of our knowledge, there are no theoretical results of subspace migration for real-world imaging an unknown anomaly when scattering parameter data is collected and affected by a small number of dipole antennas. Thus, in this study, we design an imaging algorithm based on the subspace migration technique with and without diagonal elements of the scattering matrix. Since, only the diagonal elements of the scattering matrix are affected by antennas, the result will be better than that with diagonal elements, and the designed approach will be useful in real-world microwave imaging. In order to verify this phenomenon, we carefully analyze imaging functions by establishing a relationship with an infinite series of Bessel function of integer order, the total number and location of dipole antennas, and the applied value of wavenumber. This is based on the Born approximation or asymptotic expansion formula in the presence of a small anomaly \cite{AK2} and also based on the physical factorization of the scattering matrix in the presence of an extended anomaly \cite{HSZ1}. From the identified structures of the imaging functions, we can compare the performance of traditional and designed imaging functions, discover various properties (such as the dependence of the number and location of antennas, the reason for the appearance of arbitrary and ring-type artifacts, etc.), and the fundamental limitations. In order to support our theoretical results, simulation results for small and extended anomalies using the synthetic data generated by the commercial CST STUDIO SUITE and real-data generated by a manufactured microwave machine at several angular frequencies are presented. It is worth mentioning that although the result obtained with low computational cost is good, it cannot completely determine the shapes of the anomalies. Fortunately, it can be accepted as an initial guess, and it will be possible to retrieve a better shape through an iteration process.

The remainder of this research is organized as follows: In Section \ref{sec:2}, we briefly introduce the forward problem and scattering parameter. Then, in Section \ref{sec:3}, we design imaging functions with and without the diagonal elements of the scattering matrix, analyze their structure by establishing infinite series of Bessel functions of integer order, and discover some properties of the imaging functions. In Section \ref{sec:4}, the set of simulation results using synthetic and real data is exhibited to show the effectiveness of the designed imaging functions. A conclusion including an outline of future work is given in Section \ref{sec:5}.

\section{Forward problem and scattering parameter}\label{sec:2}
In this section, we briefly introduce the forward problem and scattering parameter in the case where an anomaly $\Sigma$ with smooth boundary $\p\Sigma\in C^2$ is enclosed by dipole antennas located at $\md_n$, $n=1,2,\cdots,N$. Throughout this paper, we assume that $\Sigma$ is expressed as
\[\Sigma=\mr_\star+\rho\mB,\]
where $\mr_\star$ and $\rho$ denote the location and size of $\Sigma$, respectively, and $\mB$ is a simply connected smooth domain containing the origin. For the sake of simplicity, we assume that $\Sigma$ is a small ball centered at $\mr_\star$ with radius $\rho$, i.e., $\mB$ is a unit ball centered at the origin. In this paper, constitutive materials are fully characterized by their dielectric permittivity and and electrical conductivity at a given angular frequency $\omega$ so that the value of magnetic permeability is constant at every location $\mr$ such that $\mu(\mr)=\mu_{\mathrm{b}}=4\pi\cdot 10^{-7}\mathrm{H}\mathrm{m}^{-1}$. We denote the values of permittivity and conductivity as $\eps(\mr)$ and $\sigma(\mr)$, respectively, at location $\mr$. With this, we can introduce the piecewise constant dielectric permittivity and electrical conductivity 
\[\eps(\mr)=\left\{\begin{array}{rcl}
                     \medskip\eps_\star & \mbox{for} & \mr\in\Sigma,\\
                     \eps_{\mathrm{b}} & \mbox{for} & \mr\in\mathbb{R}^3\backslash\overline{\Sigma},
                   \end{array}
\right.
\quad\mbox{and}\quad
\sigma(\mr)=\left\{\begin{array}{rcl}
                     \medskip\sigma_\star & \mbox{for} & \mr\in\Sigma,\\
                     \sigma_{\mathrm{b}} & \mbox{for} & \mr\in\mathbb{R}^3\backslash\overline{\Sigma},
                   \end{array}
\right.\]
respectively. The values of the permittivities of $\Sigma$ and the background are $\eps_\star$ and $\eps_{\mathrm{b}}$, respectively. With this, we define the background wavenumber $k$ as
\[k=\omega^2\mu_{\mathrm{b}}\left(\eps_{\mathrm{b}}+i\frac{\sigma_{\mathrm{b}}}{\omega}\right).\]

Let $\mE_{\mathrm{inc}}(\md_n,\mr)$ be the incident electric field in a homogeneous medium due to a point current density at $\md_n$. From Maxwell's equations, $\mE_{\mathrm{inc}}(\md_n,\mr)$ satisfies
\[\nabla\times\mE_{\mathrm{inc}}(\md_n,\mr)=-i\omega\mu_{\mathrm{b}}\mH(\md_n,\mr)\quad\mbox{and}\quad\nabla\times\mH(\md_n,\mr)=(\sigma_{\mathrm{b}}+i\omega\eps_{\mathrm{b}})\mE_{\mathrm{inc}}(\md_n,\mr),\]
where $\mH$ denotes the magnetic field. Analogously, let $\mE_{\mathrm{tot}}(\mr,\md_n)$ be the total electric field in the presence of $\Sigma$. Then, $\mE_{\mathrm{tot}}(\mr,\md_n)$ satisfies
\[\nabla\times\mE_{\mathrm{tot}}(\mr,\md_n)=-i\omega\mu_{\mathrm{b}}\mH(\mr,\md_n)\quad\mbox{and}\quad\nabla\times\mH(\mr,\md_n)=(\sigma(\mr)+i\omega\eps(\mr))\mE_{\mathrm{tot}}(\mr,\md_n)\]
with the transmission condition on $\p\Sigma$. Following \cite{CZBN}, the total electric field $\mE_{\mathrm{tot}}(\mr,\md_n)$ can be represented as the following domain equation in $\Sigma$:
\begin{align}
\begin{aligned}\label{TotalFieldRepresentation}
\mE_{\mathrm{tot}}(\mr,\md_n)&=\mE_{\mathrm{inc}}(\md_n,\mr)+\mE_{\mathrm{scat}}(\mr,\md_n)\\
&=\mG(\md_n,\mr)+k^2\int_{\Sigma}\left(\frac{\eps(\mr')-\eps_{\mathrm{b}}}{\eps_{\mathrm{b}}}+i\frac{\sigma(\mr')-\sigma_{\mathrm{b}}}{\omega\sigma_{\mathrm{b}}}\right)\mG(\mr,\mr')\mE_{\mathrm{tot}}(\mr',\md_n)d\mr',
\end{aligned}
\end{align}
where $\mE_{\mathrm{scat}}(\mr,\md_n)$ denotes the scattered field and $\mG(\mr,\mr')$ is the Green's function for a uniform background (see \cite{JPH,AILP})
\[\mG(\mr,\mr')=\left(\mathbb{I}+\frac{1}{k^2}\nabla_{\mr'}\nabla_{\mr'}\right)\frac{e^{ik|\mr-\mr'|}}{4\pi|\mr-\mr'|}.\]
Here $\mathbb{I}$ denotes the $3\times3$ identity matrix and $\nabla_{\mr'}$ denotes differentiation with respect to $\mr'$. Hence, the inverse scattering problem serves to determine the parameter distribution $\eps(\mr)$ (or $\sigma(\mr)$) in a search domain of location $\mr'\in\Sigma$ from (\ref{TotalFieldRepresentation}) when $\mE_{\mathrm{inc}}(\md_n,\mr)$ is known. Notice that due to the nonlinearity and ill-posedness of the problem, reconstructing parameter distribution $\eps(\mr)$ (or $\sigma(\mr)$) without \textit{a priori} information of anomaly is very difficult. Thus, as an alternative, we focus on identifying the location $\mr'\in\Sigma$.

We denote the scattering parameter as $S(m,n)$, which is defined as
\[S(m,n):=\frac{\mathrm{V}_m^-}{\mathrm{V}_n^+},\]
where
$\mathrm{V}_m^-$ and $\mathrm{V}_n^+$ denote the output voltage (or reflected waves) at the $m-$th antenna and the input voltage (or incident waves) at $n-$th antenna, respectively. We let $S_{\mathrm{tot}}(m,n)$ and $S_{\mathrm{inc}}(m,n)$ be the total field and incident-field $S-$parameters (i.e., the measured $S-$parameters with and without $\Sigma$), respectively. Let $S_{\mathrm{scat}}(m,n)$ be the scattered-field $S-$parameter obtained by subtracting $S_{\mathrm{tot}}(m,n)$ and $S_{\mathrm{inc}}(m,n)$. Then, with the existence of $\Sigma$, $S_{\mathrm{scat}}(m,n)$ can be represented as follows (see \cite{HSM2}):
\begin{equation}\label{Sparameter}
S_{\mathrm{scat}}(m,n)=\frac{ik^2}{4\omega\mu_{\mathrm{b}}}\int_{\Sigma}\left(\frac{\eps(\mr)-\eps_{\mathrm{b}}}{\eps_{\mathrm{b}}}+i\frac{\sigma(\mr)-\sigma_{\mathrm{b}}}{\omega\sigma_{\mathrm{b}}}\right)\mE_{\mathrm{inc}}(\md_n,\mr)\cdot\mE_{\mathrm{tot}}(\mr,\md_m)d\mr.
\end{equation}
This representation will play a key role in the design of an imaging function based on subspace migration. Notice that in many studies, the imaging algorithm has been designed under the assumption that measurement data $S_{\mathrm{scat}}(m,n)$ is affected by anomaly $\Sigma$ only. Here, as we verified in Figure \ref{AntennaEffects}, we assume that $S_{\mathrm{scat}}(n,n)$ is affected not only by $\Sigma$, but also by $\md_n$, $n=1,2,\cdots,N$.

\section{Imaging function with and without diagonal elements of scattering matrix: introduction and analysis}\label{sec:3}
\subsection{Introduction to imaging functions}
In this section, we introduce imaging function for detecting anomaly from generated scattering matrices whose elements are measured scattered-field $S-$parameters. We separately consider the following two cases: with and without the diagonal elements of the scattering matrix.

\subsubsection{Imaging function with diagonal elements} First, we consider the following scattering matrix $\mathbb{F}$ with diagonal elements:
\begin{equation}\label{MSR}
\mathbb{F}=\left[
               \begin{array}{ccccc}
                 \medskip S_{\mathrm{scat}}(1,1) & S_{\mathrm{scat}}(1,2) & \cdots & S_{\mathrm{scat}}(1,N-1) & S_{\mathrm{scat}}(1,N) \\
                 \medskip S_{\mathrm{scat}}(2,1) & S_{\mathrm{scat}}(2,2) & \cdots & S_{\mathrm{scat}}(2,N-1) & S_{\mathrm{scat}}(2,N) \\
                 \medskip \vdots & \vdots & \ddots & \vdots & \vdots \\
                 S_{\mathrm{scat}}(N,1) & S_{\mathrm{scat}}(N,2) & \cdots & S_{\mathrm{scat}}(N,N-1) & S_{\mathrm{scat}}(N,N)
               \end{array}
             \right].
\end{equation}

Note that the wavelength $\lambda$ corresponding to the wavenumber $k$ is larger than the radius of anomaly $\Sigma$. This problem can be viewed as that of imaging a small target and, by applying the Born approximation \cite{HSM2}, formula (\ref{Sparameter}) can be approximated as follows:
\begin{align}
\begin{aligned}\label{Formula-S2}
S_{\mathrm{scat}}(m,n)&\approx\frac{ik^2}{4\omega\mu_{\mathrm{b}}}\int_{\Sigma}\left(\frac{\eps(\mr)-\eps_{\mathrm{b}}}{\eps_{\mathrm{b}}}+i\frac{\sigma(\mr)-\sigma_{\mathrm{b}}}{\omega\sigma_{\mathrm{b}}}\right)\mE_{\mathrm{inc}}(\md_n,\mr)\cdot\mE_{\mathrm{inc}}(\mr,\md_m)d\mr\\
&\approx=\rho^2\frac{ik^2\pi}{4\omega\mu_{\mathrm{b}}}\left(\frac{\eps_\star-\eps_{\mathrm{b}}}{\eps_{\mathrm{b}}}+i\frac{\sigma_\star-\sigma_{\mathrm{b}}}{\omega\sigma_{\mathrm{b}}}\right)\mE_{\mathrm{inc}}(\md_n,\mr_\star)\cdot\mE_{\mathrm{inc}}(\mr_\star,\md_m).
\end{aligned}
\end{align}
This means that the range of $\mathbb{F}$ can then be determined from the span of
\begin{equation}\label{SpanVectors}
\bigg[\mE_{\mathrm{inc}}(\md_1,\mr_\star),\mE_{\mathrm{inc}}(\md_2,\mr_\star),\cdots,\mE_{\mathrm{inc}}(\md_N,\mr_\star)\bigg]^{\mathtt{T}}
\end{equation}
corresponding to $\Sigma$.

On the basis of this observation, we introduce the imaging technique. Let us perform singular-value decomposition (SVD) on $\mathbb{F}$. Notice that, since the elements of $\mathbb{F}$ are influenced by the anomaly and antennas, there are several nonzero singular values (see Figure \ref{Distribution}). Thus, SVD of $\mathbb{F}$ can be written as
\begin{equation}\label{SVD}
\mathbb{F}=\mathbb{USV}^*=\sum_{m=1}^{N}\tau_m\mU_m\mV_m^*\approx\sum_{m=1}^{M}\tau_m\mU_m\mV_m^*,
\end{equation}
where $\tau_m$ are the singular values, $\mU_m$ and $\mV_m$ are respectively the left and right singular vectors of $\mathbb{F}$, and $\mathbb{S}$ is a real nonnegative diagonal matrix with components $\tau_1,\tau_2,\cdots,\tau_N$ satisfying
\begin{equation}\label{SingularValue}
\tau_1\geq\tau_2\geq\tau_3\geq\cdots\geq\tau_M>0\quad\mbox{and}\quad\tau_{M+1},\tau_{M+2},\cdots,\tau_N\approx0.
\end{equation}
Then, on the basis of (\ref{SpanVectors}), (\ref{SVD}), and (\ref{SingularValue}), we can introduce the following imaging function ; for a search point $\mr$,
\begin{equation}\label{ImagingFunctionFull}
\mathfrak{F}_{\mathrm{Full}}(\mr):=\left|\sum_{m=1}^{M}\left\langle\frac{\mW(\mr)}{|\mW(\mr)|},\mU_m\right\rangle\left\langle\frac{\mW(\mr)}{|\mW(\mr)|},\overline{\mV}_m\right\rangle\right|,
\end{equation}
where $\langle,\rangle$ denotes the inner product $\langle\mU,\mV\rangle:=\mU^*\mV$, $\overline{\mV}$ is the complex conjugate of $\mV$, superscript $*$ is the mark of Hermitian, and
\begin{equation}\label{TestVector}
\mW(\mr)=\bigg[\mE_{\mathrm{inc}}(\md_1,\mr),\mE_{\mathrm{inc}}(\md_2,\mr),\cdots,\mE_{\mathrm{inc}}(\md_N,\mr)\bigg]^{\mathtt{T}}.
\end{equation}
Then, the plot of $\mathfrak{F}_{\mathrm{Full}}(\mr)$ is expected to exhibit peaks of magnitude $1$ at $\mr=\mr_\star\in\Sigma$ and small magnitude at $\mr\notin\Sigma$. For a detailed discussion, we refer to \cite{AGKPS,P-SUB3}.

\begin{rem}[Application of the asymptotic expansion formula]
If the size of the anomaly is small it is possible to apply the asymptotic expansion formula from \cite{AILP,AVV} instead of the Born approximation
\[\mE_{\mathrm{tot}}(\mr,\md_n)\approx\mE_{\mathrm{inc}}(\mr,\md_n)+o(\rho^3),\]
and the imaging function can be designed in a similar manner. We refer to \cite{SKL} and Section \ref{sec:4} for a detailed discussion.
\end{rem}

\begin{rem}[Imaging of an extended anomaly]\label{RemarkExtended}
If the radius $\rho$ of $\Sigma$ is not small, we cannot apply the Born approximation, making it impossible to completely obtain the form of $\Sigma$. Note that, based on \cite{HSZ1}, $\mathbb{F}$ can be decomposed as
\begin{equation}\label{Factorization}
  \mathbb{F}=\frac{ik^2}{4\omega\mu_{\mathrm{b}}}\int_{\Sigma}\left(\frac{\eps(\mr)-\eps_{\mathrm{b}}}{\eps_{\mathrm{b}}}+i\frac{\sigma(\mr)-\sigma_{\mathrm{b}}}{\omega\sigma_{\mathrm{b}}}\right)\mathbb{E}_{\mathrm{inc}}(\mr)\mathbb{E}_{\mathrm{tot}}(\mr)d\mr,
  \end{equation}
  where
  \[\mathbb{E}_{\mathrm{inc}}(\mr)=\bigg[\mE_{\mathrm{inc}}(\md_1,\mr),\mE_{\mathrm{inc}}(\md_2,\mr),\cdots,\mE_{\mathrm{inc}}(\md_N,\mr)\bigg]^{\mathtt{T}}\]
  and
  \[\mathbb{E}_{\mathrm{tot}}(\mr)=\bigg[\mE_{\mathrm{tot}}(\mr,\md_1),\mE_{\mathrm{tot}}(\mr,\md_2),\cdots,\mE_{\mathrm{tot}}(\mr,\md_N)\bigg],\]
 respectively. Formula (\ref{Factorization}) gives a physical factorization of $\mathbb{F}$ that separates the known incident field from the unknown total field. Based on this factorization, the range of $\mathbb{F}$ is determined based on the span of $\mathbb{E}_{\mathrm{inc}}(\mr)$ corresponding to $\mr_\star\in\p\Sigma$. This means that the imaging function $\mathfrak{F}_{\mathrm{Full}}(\mr)$ of (\ref{ImagingFunctionFull}) can be defined by selecting the first $M-$singular vectors of $\mathbb{F}$, and the outline of $\p\Sigma$ will be recognized via the map of $\mathfrak{F}_{\mathrm{Full}}(\mr)$.
\end{rem}

\subsubsection{Imaging function without diagonal elements}
In some real-world applications, it is very hard to measure $S_{\mathrm{scat}}(n,n)$ for $n=1,2,\cdots,N$, because each of the $N$ antennas is used for signal transmission and the remaining $N-1$ antennas are used for signal reception. Therefore, we can use the following matrix data:
\begin{equation}\label{MSRunknown}
\left[
                 \begin{array}{ccccc}
                 \medskip \mbox{unknown} & S_{\mathrm{scat}}(1,2) & \cdots & S_{\mathrm{scat}}(1,N-1) & S_{\mathrm{scat}}(1,N) \\
                 \medskip S_{\mathrm{scat}}(2,1) & \mbox{unknown} & \cdots & S_{\mathrm{scat}}(2,N-1) & S_{\mathrm{scat}}(2,N) \\
                 \medskip \vdots & \vdots & \ddots & \vdots & \vdots \\
                 S_{\mathrm{scat}}(N,1) & S_{\mathrm{scat}}(N,2) & \cdots & S_{\mathrm{scat}}(N,N-1) & \mbox{unknown}
               \end{array}
             \right].
\end{equation}
Thus, it is impossible to identify the range of above matrix. Due to this reason, by setting $S_{\mathrm{scat}}(n,n)\equiv0$ for $n=1,2,\cdots,N$, we consider the corresponding scattering matrix $\mathbb{D}$:
\begin{equation}\label{MSRWithout}
\mathbb{D}=\left[
               \begin{array}{ccccc}
                 \medskip 0 & S_{\mathrm{scat}}(1,2) & \cdots & S_{\mathrm{scat}}(1,N-1) & S_{\mathrm{scat}}(1,N) \\
                 \medskip S_{\mathrm{scat}}(2,1) & 0 & \cdots & S_{\mathrm{scat}}(2,N-1) & S_{\mathrm{scat}}(2,N) \\
                 \medskip \vdots & \vdots & \ddots & \vdots & \vdots \\
                 S_{\mathrm{scat}}(N,1) & S_{\mathrm{scat}}(N,2) & \cdots & S_{\mathrm{scat}}(N,N-1) & 0
               \end{array}
             \right].
\end{equation}

We apply the traditional subspace migration technique as follows. Since the elements of $\mathbb{D}$ are influenced by the anomaly only, there is only one nonzero singular value (see Figure \ref{Distribution}). Thus, the SVD of $\mathbb{D}$ can be written as
\[\mathbb{D}=\mathbb{USV}^*=\sum_{m=1}^{N}\tau_m\mU_m\mV_m^*\approx\tau_1\mU_1\mV_1^*\]
and correspondingly, we can introduce the imaging function
\begin{equation}\label{ImagingFunctionWithout}
\mathfrak{F}_{\mathrm{Diag}}(\mr):=\left|\left\langle\frac{\mW(\mr)}{|\mW(\mr)|},\mU_1\right\rangle\left\langle\frac{\mW(\mr)}{|\mW(\mr)|},\overline{\mV}_1\right\rangle\right|,
\end{equation}
where $\mW(\mr)$ is defined in (\ref{TestVector}). Then, the plot of $\mathfrak{F}_{\mathrm{Diag}}(\mr)$ is expected to exhibit peaks of magnitude $1$ at $\mr=\mr_\star\in\Sigma$ and small magnitude at $\mr\notin\Sigma$. For imaging of a large anomaly, we apply the same strategy.

\subsection{Analysis of imaging functions}
Here, we carefully analyze the mathematical structure of imaging functions $\mathfrak{F}_{\mathrm{Full}}(\mr)$ and $\mathfrak{F}_{\mathrm{Diag}}(\mr)$. Notice that, from the mathematical treatment of the scattering of time-harmonic electromagnetic waves from thin infinitely long cylindrical obstacles (i.e., two-dimensional transverse-magnetic polarization), we can observe that
\[\mE_{\mathrm{inc}}(\md,\mr)=\mG(\md,\mr)=-\frac{i}{4}H_0^{(1)}(k|\md-\mr|),\]
where $H_0^{(1)}$ denotes the Hankel function of order zero of the first kind. We refer to \cite[Fig. 1]{JPH} for illustration. Then, we can obtain the following result:

\begin{thm}[Mathematical structure of imaging functions]\label{TheoremStructure}
Assume that the total number of antennas, $N$, is small. Let $\vt_n=\md_n/|\md_n|=[\cos\theta_n,\sin\theta_n]^{\mathtt{T}}$ and $\mr-\mr_\star=|\mr-\mr_\star|[\cos\phi,\sin\phi]^{\mathtt{T}}$. Then, $\mathfrak{F}_{\mathrm{Full}}(\mr)$ and $\mathfrak{F}_{\mathrm{Diag}}(\mr)$ can be represented as follows: if $\mr$ and $\md_n$ satisfies $|\md_n-\mr|,|\md_n-\mr_\star|\gg0.25/k$
\begin{align}
\label{StructureFull}\mathfrak{F}_{\mathrm{Full}}(\mr)&=\left|\left(J_0(k|\mr-\mr_\star|)+\frac{1}{N}\sum_{n=1}^{N}\Psi_1(k,n)+\sum_{m=2}^{M}J_0(k|\mr-\mr_m|)+\frac{1}{N}\sum_{n=1}^{N}\sum_{m=2}^{M}\Psi_2(k,n)\right)^2\right|,\\
\label{StructureDiag}\mathfrak{F}_{\mathrm{Diag}}(\mr)&=\frac{N}{N-1}\left|\left(J_0(k|\mr-\mr_\star|)+\frac{1}{N}\sum_{n=1}^{N}\Psi_1(k,n)\right)^2-\frac{1}{N}\left(J_0(2k|\mr-\mr_\star|)+\frac{1}{N}\sum_{n=1}^{N}\Psi_1(2k,n)\right)\right|,
\end{align}
where
\[\Psi_1(k,n)=\sum_{s\in\mathbb{Z}^*\backslash\set{0}}i^s J_{s}(k|\mr-\mr_\star|)e^{is(\theta_n-\phi)}\quad\mbox{and}\quad\Psi_2(k,n)=\sum_{s\in\mathbb{Z}^*\backslash\set{0}}i^s J_{s}(k|\mr-\mr_m|)e^{is(\theta_n-\phi_m)}.\]
Here, $\mr_m\ne\mr_\star$, $m=2,3,\cdots,M$ denotes an unknown location that depends on the locations of $\Sigma$ and $\md_n$.
\end{thm}
\begin{proof}
First, we derive (\ref{StructureFull}). Since $|\md_n-\mr|\gg0.25/k$ and $|\md_n-\mr_\star|\gg0.25/k$, applying the asymptotic form of the Hankel function
\[H_0^{(1)}(k|\md_n-\mr|)=\frac{1+i}{4\sqrt{k\pi}}\frac{e^{ik|\md_n|}}{\sqrt{|\md_n|}}e^{-ik\vt_n\cdot\mr}+o\left(\frac{1}{\sqrt{|\md_n|}}\right)\]
yields
\[\mathbf{U}_1\approx\overline{\mathbf{V}}_1\approx\frac{\mathbf{W}(\mr_\star)}{|\mathbf{W}(\mr_\star)|}=\frac{1}{\sqrt{N}}\bigg[e^{-ik\vt_1\cdot\mr_\star},e^{-ik\vt_2\cdot\mr_\star},\cdots,e^{-ik\vt_N\cdot\mr_\star}\bigg]^{\mathtt{T}}\]
and
\[\mathbf{U}_m\approx\overline{\mathbf{V}}_m\approx\frac{\mathbf{W}(\mr_m)}{|\mathbf{W}(\mr_m)|}=\frac{1}{\sqrt{N}}\bigg[e^{-ik\vt_1\cdot\mr_m},e^{-ik\vt_2\cdot\mr_m},\cdots,e^{-ik\vt_N\cdot\mr_m}\bigg]^{\mathtt{T}},\]
for $m=2,3,\cdots,M$. Then, we can observe that
\begin{align*}
\sum_{m=1}^{M}\left\langle\frac{\mW(\mr)}{|\mW(\mr)|},\mU_m\right\rangle\left\langle\frac{\mW(\mr)}{|\mW(\mr)|},\overline{\mV}_m\right\rangle=&\left\langle\frac{\mW(\mr)}{|\mW(\mr)|},\frac{\mW(\mr_\star)}{|\mW(\mr_\star)|}\right\rangle\left\langle\frac{\mW(\mr)}{|\mW(\mr)|},\frac{\mW(\mr_\star)}{|\mW(\mr_\star)|}\right\rangle\\
&+\sum_{m=2}^{M}\left\langle\frac{\mW(\mr)}{|\mW(\mr)|},\frac{\mW(\mr_m)}{|\mW(\mr_m)|}\right\rangle\left\langle\frac{\mW(\mr)}{|\mW(\mr)|},\frac{\mW(\mr_m)}{|\mW(\mr_m)|}\right\rangle.
\end{align*}
Now, let $\mathbb{Z}^*:=\mathbb{Z}\cup\set{-\infty,\infty}$. Since the following Jacobi-Anger expansion holds uniformly,
\begin{equation}\label{JacobiAnger}
e^{ix\cos\theta}=J_0(x)+\sum_{s\in\mathbb{Z}^*\backslash\set{0}}i^s J_{s}(x)e^{is\theta},
\end{equation}
we can evaluate
\begin{align}
\begin{aligned}\label{term1}
\left\langle\frac{\mW(\mr)}{|\mW(\mr)|},\frac{\mW(\mr_\star)}{|\mW(\mr_\star)|}\right\rangle&\left\langle\frac{\mW(\mr)}{|\mW(\mr)|},\frac{\mW(\mr_\star)}{|\mW(\mr_\star)|}\right\rangle=\left(\frac{1}{N}\sum_{n=1}^{N}e^{ik\vt_n\cdot(\mr-\mr_\star)}\right)^2=\left(\frac{1}{N}\sum_{n=1}^{N}e^{ik|\mr-\mr_\star|\cos(\theta_n-\phi)}\right)^2\\
&=\left(\frac{1}{N}\sum_{n=1}^{N}\left(J_0(k|\mr-\mr_\star|)+\sum_{s\in\mathbb{Z}^*\backslash\set{0}}i^s J_{s}(k|\mr-\mr_\star|)e^{is(\theta_n-\phi)}\right)\right)^2\\
&=\left(J_0(k|\mr-\mr_\star|)+\frac{1}{N}\sum_{n=1}^{N}\sum_{s\in\mathbb{Z}^*\backslash\set{0}}i^s J_{s}(k|\mr-\mr_\star|)e^{is(\theta_n-\phi)}\right)^2.
\end{aligned}
\end{align}
By letting $\mr-\mr_m=|\mr-\mr_m|[\cos\phi_m,\sin\phi_m]^{\mathtt{T}}$,
\begin{equation}\label{term2}
\left\langle\frac{\mW(\mr)}{|\mW(\mr)|},\frac{\mW(\mr_m)}{|\mW(\mr_m)|}\right\rangle\left\langle\frac{\mW(\mr)}{|\mW(\mr)|},\frac{\mW(\mr_m)}{|\mW(\mr_m)|}\right\rangle
=\left(J_0(k|\mr-\mr_m|)+\frac{1}{N}\sum_{n=1}^{N}\sum_{s\in\mathbb{Z}^*\backslash\set{0}}i^s J_{s}(k|\mr-\mr_m|)e^{is(\theta_n-\phi_m)}\right)^2.
\end{equation}
Combining (\ref{term1}) and (\ref{term2}), we obtain (\ref{StructureFull}).

Next, we derive (\ref{StructureDiag}). In this case,
$\mathbb{D}$ can be written as
\[\mathbb{D}\approx\rho_1\mU_1\overline{\mV}_1^{\mathtt{T}}\propto\frac{1}{N-1}
\left[\begin{array}{cccc}
\medskip 0&e^{-ik(\vt_1+\vt_2)\cdot\mr_\star}&\cdots&e^{-ik(\vt_1+\vt_N)\cdot\mr_\star}\\
e^{-ik(\vt_2+\vt_1)\cdot\mr_\star}&0&\cdots&e^{-ik(\vt_2+\vt_N)\cdot\mr_\star}\\
\medskip \vdots&\vdots&\ddots&\vdots\\
e^{-ik(\vt_N+\vt_1)\cdot\mr_\star}&e^{-ik(\vt_N+\vt_2)\cdot\mr_\star}&\cdots&0\end{array}\right].\]
Then, we can calculate
\begin{align*}
&\left\langle\frac{\mW(\mr)}{|\mW(\mr)|},\mU_1\right\rangle\left\langle\frac{\mW(\mr)}{|\mW(\mr)|},\overline{\mV}_1\right\rangle=\frac{\mW(\mr)^*}{|\mW(\mr)^*|}\mU_1\overline{\mV}_1^{\mathtt{T}}\frac{\overline{\mW}(\mr_\star)}{|\overline{\mW}(\mr_\star)|}\\
&=\frac{1}{N(N-1)}\left[\begin{array}{c}
\medskip e^{ik\vt_1\cdot\mr}\\
\medskip e^{ik\vt_2\cdot\mr}\\
\medskip\vdots\\
e^{ik\vt_N\cdot\mr}
\end{array}\right]^{\mathtt{T}}
\left[\begin{array}{cccc}
\medskip 0&e^{-ik(\vt_1+\vt_2)\cdot\mr_\star}&\cdots&e^{-ik(\vt_1+\vt_N)\cdot\mr_\star}\\
e^{-ik(\vt_2+\vt_1)\cdot\mr_\star}&0&\cdots&e^{-ik(\vt_2+\vt_N)\cdot\mr_\star}\\
\medskip \vdots&\vdots&\ddots&\vdots\\
e^{-ik(\vt_N+\vt_1)\cdot\mr_\star}&e^{-ik(\vt_N+\vt_2)\cdot\mr_\star}&\cdots&0\end{array}\right]
\left[\begin{array}{c}
\medskip e^{ik\vt_1\cdot\mr}\\
\medskip e^{ik\vt_2\cdot\mr}\\
\medskip\vdots\\
e^{ik\vt_N\cdot\mr}
\end{array}\right]\\
&=\frac{1}{N-1}(\Phi_1-\Phi_2),
\end{align*}
where
\[
\Phi_1=\frac{1}{N}\left[\begin{array}{c}
\medskip e^{ik\vt_1\cdot\mr}\\
\medskip e^{ik\vt_2\cdot\mr}\\
\medskip\vdots\\
e^{ik\vt_N\cdot\mr}
\end{array}\right]^{\mathtt{T}}
\left[\begin{array}{cccc}
\medskip e^{-ik(\vt_1+\vt_1)\cdot\mr_\star}&e^{-ik(\vt_1+\vt_2)\cdot\mr_\star}&\cdots&e^{-ik(\vt_1+\vt_N)\cdot\mr_\star}\\
e^{-ik(\vt_2+\vt_1)\cdot\mr_\star}&e^{-ik(\vt_2+\vt_2)\cdot\mr_\star}&\cdots&e^{-ik(\vt_2+\vt_N)\cdot\mr_\star}\\
\medskip \vdots&\vdots&\ddots&\vdots\\
e^{-ik(\vt_N+\vt_1)\cdot\mr_\star}&e^{-ik(\vt_N+\vt_2)\cdot\mr_\star}&\cdots&e^{-ik(\vt_N+\vt_N)\cdot\mr_\star}\end{array}\right]
\left[\begin{array}{c}
\medskip e^{ik\vt_1\cdot\mr}\\
\medskip e^{ik\vt_2\cdot\mr}\\
\medskip\vdots\\
e^{ik\vt_N\cdot\mr}
\end{array}\right]\]
and
\[
\Phi_2=\frac{1}{N}\left[\begin{array}{c}
\medskip e^{ik\vt_1\cdot\mr}\\
\medskip e^{ik\vt_2\cdot\mr}\\
\medskip\vdots\\
e^{ik\vt_N\cdot\mr}
\end{array}\right]^{\mathtt{T}}
\left[\begin{array}{cccc}
\medskip e^{-ik(\vt_1+\vt_1)\cdot\mr_\star}&0&\cdots&0\\
0&e^{-ik(\vt_2+\vt_2)\cdot\mr_\star}&\cdots&0\\
\medskip \vdots&\vdots&\ddots&\vdots\\
0&0&\cdots&e^{-ik(\vt_N+\vt_N)\cdot\mr_\star}\end{array}\right]
\left[\begin{array}{c}
\medskip e^{ik\vt_1\cdot\mr}\\
\medskip e^{ik\vt_2\cdot\mr}\\
\medskip\vdots\\
e^{ik\vt_N\cdot\mr}
\end{array}\right].\]

Based on the derivation of (\ref{StructureFull}), we can observe that
\begin{align}
\begin{aligned}\label{term3}
\Phi_1&=N\left(\frac{1}{N}\sum_{n=1}^{N}\left(J_0(k|\mr-\mr_\star|)+\sum_{s\in\mathbb{Z}^*\backslash\set{0}}i^s J_{s}(k|\mr-\mr_\star|)e^{is(\theta_n-\phi)}\right)\right)^2\\
&=N\left(J_0(k|\mr-\mr_\star|)+\frac{1}{N}\sum_{n=1}^{N}\Psi_1(k,n)\right)^2.
\end{aligned}
\end{align}
Furthermore, applying (\ref{JacobiAnger}) again, we can obtain
\begin{align}
\begin{aligned}\label{term4}
\Phi_2&=\frac{1}{N}\sum_{n=1}^{N}e^{2ik\vt_n\cdot(\mr-\mr_\star)}=\frac{1}{N}\sum_{n=1}^{N}\left(J_0(2k|\mr-\mr_\star|)+\sum_{s\in\mathbb{Z}^*\backslash\set{0}}i^s J_{s}(2k|\mr-\mr_\star|)e^{is(\theta_n-\phi)}\right)\\
&=J_0(2k|\mr-\mr_\star|)+\frac{1}{N}\sum_{n=1}^{N}\sum_{s\in\mathbb{Z}^*\backslash\set{0}}i^s J_{s}(2k|\mr-\mr_\star|)e^{is(\theta_n-\phi)}=J_0(2k|\mr-\mr_\star|)+\frac{1}{N}\sum_{n=1}^{N}\Psi_1(2k,n).
\end{aligned}
\end{align}
Finally, by combining (\ref{term3}) and (\ref{term4}), we can obtain (\ref{StructureDiag}). This completes the proof.
\end{proof}

\subsection{Discovered properties of imaging functions}\label{Remark}
Following identification of the structures (\ref{StructureFull}) and (\ref{StructureDiag}), we find the following properties of imaging functions.
\begin{enumerate}
\item Since $J_0(0)=1$ and $J_s(0)=0$ for $s=2,3,\cdots,$ $\mathfrak{F}_{\mathrm{Full}}(\mr)=1$ and $\mathfrak{F}_{\mathrm{Diag}}(\mr)=1$ at $\mr=\mr_\star$. This is why the outline of a small anomaly can be imaged via $\mathfrak{F}_{\mathrm{Full}}(\mr)$ and $\mathfrak{F}_{\mathrm{Diag}}(\mr)$.
\item Maps of $\mathfrak{F}_{\mathrm{Full}}(\mr)$ and $\mathfrak{F}_{\mathrm{Diag}}(\mr)$ will contain some artifacts due to the terms $\Psi_1$ and $\Psi_2$. In particular, due to the term $J_0(k|\mr-\mr_m|)$ of (\ref{StructureFull}), the map of $\mathfrak{F}_{\mathrm{Full}}(\mr)$ will contain more unexpected artifacts than that of $\mathfrak{F}_{\mathrm{Diag}}(\mr)$ (see Figure \ref{Result}). This is the theoretical reason why eliminating the diagonal terms of $\mathbb{F}$ guarantees a better result.
\item The total number of dipole antennas significantly contributes to the imaging performance. If one can increase the total number of antennas $N$, the effects of $\Psi_1$ and $\Psi_2$ become negligible such that the artifacts in the maps of $\mathfrak{F}_{\mathrm{Full}}(\mr)$ and $\mathfrak{F}_{\mathrm{Diag}}(\mr)$ will be reduced. Nevertheless, the map of $\mathfrak{F}_{\mathrm{Full}}(\mr)$ still contains some artifacts due to the $J_0(k|\mr-\mr_m|)$ term.
\item Just as in the previous contribution, the quality of imaging results depends on the value of applied frequency. If one applies low frequency, due to the oscillating property of Bessel functions, small amounts of artifacts will be included in the map of $\mathfrak{F}_{\mathrm{Diag}}(\mr)$, but the resolution will be poor. Otherwise, if one applies high frequency, one will obtain a good result, but more artifacts will be included in the map of $\mathfrak{F}_{\mathrm{Diag}}(\mr)$.
\end{enumerate}

\section{Simulation results using synthetic and real data}\label{sec:4}
In order to support the identified structures in Theorem \ref{TheoremStructure} and compare imaging performances, we exhibit a set of simulation results. For this, we use $N=16$ dipole antennas $\md_n$ such that
\[\md_n=0.09\mathrm{m}[\cos\theta_n,\sin\theta_n]^{\mathtt{T}},\quad\theta_n=\frac{3\pi}{2}-\frac{2\pi(n-1)}{N},\]
and we apply an $f=1$ GHz frequency, i.e., $\omega=2\pi f=2\pi\cdot10^9$ Rad/s angular frequency. To perform the imaging, we set the search domain $\Omega$ to $\Omega=0.1\mathrm{m}[-1,1]^{\mathtt{T}}\times0.1\mathrm{m}[-1,1]^{\mathtt{T}}$. For each $\mr\in\Omega$, the step size of $\mr$ is considered to be in the order of $0.001\mathrm{m}$. The material properties, locations, and sizes of anomalies, as well as the background, are written in Table \ref{Table1}. Notice that for the anomaly $\Sigma_\mathrm{S}$, since $\mu_0$ is constant, $\eps_\star=55$, and its diameter $d$ is $d=0.02\mathrm{m}$, we can observe that
\[\mbox{refractive index}\times\mbox{diameter}=\sqrt{\frac{\eps_\star}{\eps_{\mathrm{b}}}}\times d=0.0332<0.0670=\lambda=\frac{2\pi}{k}.\]
Hence, based on \cite{SKL}, $\Sigma_\mathrm{S}$ can be regarded as a small anomaly so that application of the Born approximation is valid for imaging. For the anomaly $\Sigma_\mathrm{E}$, since
\[\mbox{refractive index}\times\mbox{diameter}=\sqrt{\frac{\eps_\star}{\eps_{\mathrm{b}}}}\times d=0.0866>0.0670=\lambda=\frac{2\pi}{k},\]
it cannot be regarded as a small anomaly. Hence, motivated from \cite{HSZ1}, we let $\Sigma_\mathrm{E}$ be an extended anomaly and only its boundary can be identified via the map of $\mathfrak{F}_{\mathrm{Full}}(\mr)$ and $\mathfrak{F}_{\mathrm{Diag}}(\mr)$, refer to Remark \ref{RemarkExtended}.

It is worth mentioning that all $S-$parameters in Sections \ref{Section-Compare}, \ref{Section-Antenna}, and \ref{Section-Extended} were generated using CST STUDIO SUITE, and that those in the Sections \ref{Section-Real} and \ref{Section-Real-Limit} were generated using the developed microwave machine.

\begin{table}[h]
\begin{center}
\begin{tabular}{c||c|c|c|c}
\hline\hline  \centering Target&Permittivity&Conductivity (S/m)&~~~~Location~~~~&Radius\\
\hline\hline\centering Background&$20$&$0.2$&$-$&$-$\\
\hline\centering Anomaly $\Sigma_\mathrm{S}$&$55$&$1.2$&$0.01\mathrm{m}[1,3]^{\mathtt{T}}$&$0.010\mathrm{m}$\\
\hline\centering Anomaly $\Sigma_\mathrm{E}$&$15$&$0.5$&$0.01\mathrm{m}[1,2]^{\mathtt{T}}$&$0.050\mathrm{m}$\\
\hline\hline
\end{tabular}
\caption{\label{Table1}Electromagnetic properties, locations, and sizes of anomalies as well as background}
\end{center}
\end{table}

\begin{figure}[h]
\begin{center}
\includegraphics[width=0.495\textwidth]{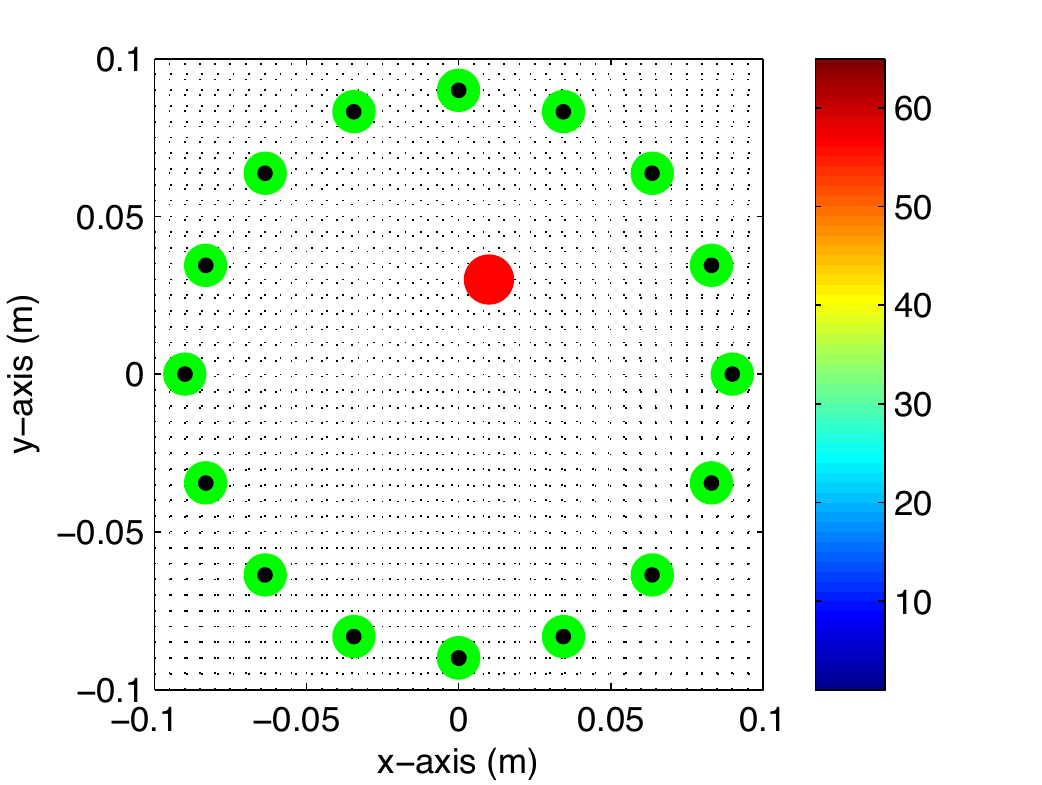}
\includegraphics[width=0.495\textwidth]{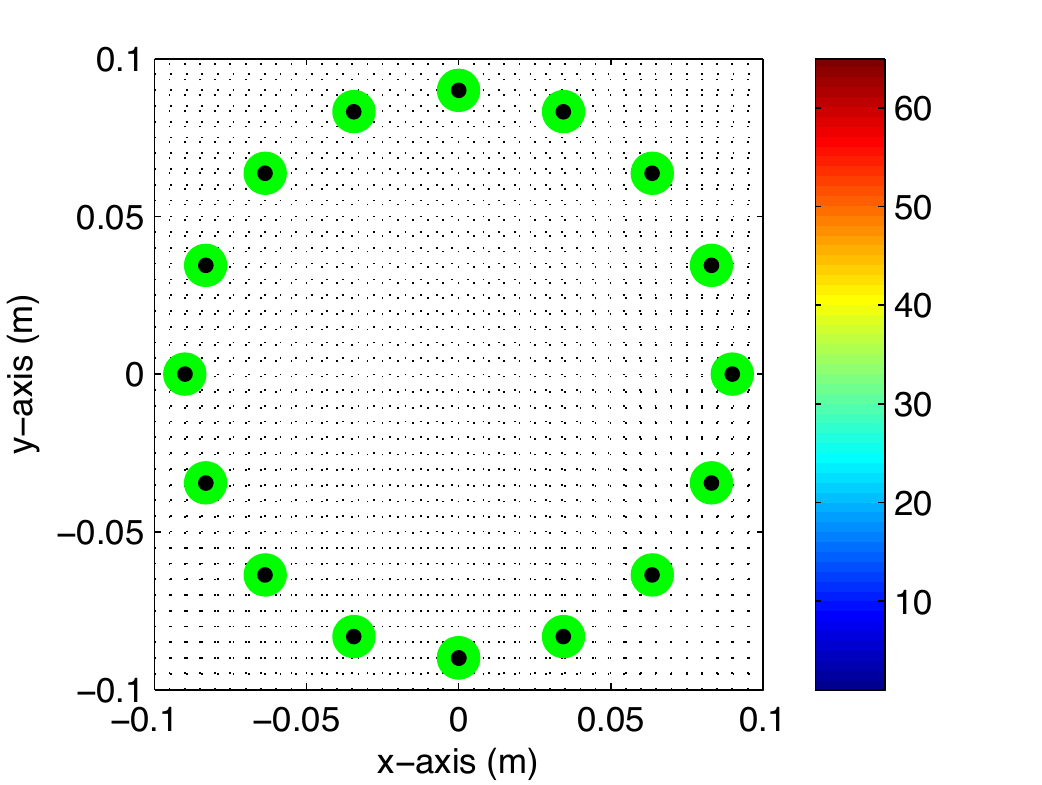}
\caption{\label{Configuration}Test configurations with (left) and without (right) an anomaly.}
\end{center}
\end{figure}

\subsection{Comparison between $\mathfrak{F}_{\mathrm{Full}}(\mr)$ and $\mathfrak{F}_{\mathrm{Diag}}(\mr)$}\label{Section-Compare}
Figure \ref{Distribution} shows the distributions of singular values of $\mathbb{F}$ and $\mathbb{D}$. Notice that, since there exists only one small anomaly, the number of nonzero singular values of $\mathbb{F}$ must be equal to $1$. However, as we discussed in Section \ref{sec:3}, due to the existence of anomalies and antennas, the first five singular values can be regarded as nonzero, such that the first five singular vectors associated with these values can be used to set the imaging function $\mathfrak{F}_{\mathrm{Full}}(\mr)$. By contrast, the number of nonzero singular values of $\mathbb{D}$ can be regarded as $1$. Thus, we can establish an imaging function $\mathfrak{F}_{\mathrm{Diag}}(\mr)$ with the first singular vector linked to the first singular value.

Figure \ref{Result} shows maps of $\mathfrak{F}_{\mathrm{Full}}(\mr)$ and $\mathfrak{F}_{\mathrm{Diag}}(\mr)$. Although the outline of $\Sigma$ can be identified, arbitrary unexpected artifacts can also be included in the map. Based on the identified structure of (\ref{StructureFull}), we can say that this is due to the existence of disturbance terms (i.e., the last three terms from the right). This result supports identified structure (\ref{StructureFull}), explaining why the outline shape of the anomaly and the unexpected artifacts are included in the map and explaining why it is quietly different from the previous results in \cite{AGKPS,P-SUB3,P-SUB5}. Fortunately, when we use the imaging function $\mathfrak{F}_{\mathrm{Diag}}(\mr)$, the shape of $\Sigma$ can be obtained clearly without disturbing unexpected artifacts. We also refer to contour lines in Figure \ref{Contourplotsmall} to compare the appearance of artifacts. Notice that ring-shaped artifacts centered at location $\mr_\star$ are included in the map of $\mathfrak{F}_{\mathrm{Diag}}(\mr)$, which is consistent with the identified structure (\ref{StructureDiag}).

\begin{figure}[h]
\begin{center}
\includegraphics[width=0.495\textwidth]{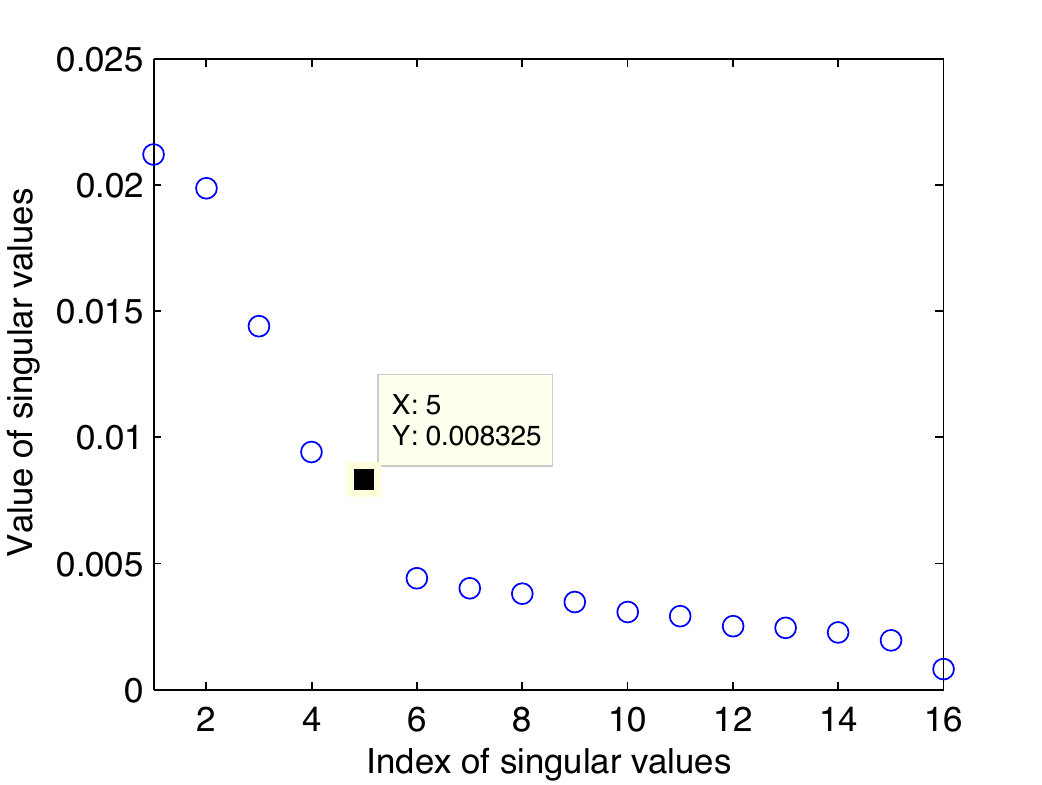}
\includegraphics[width=0.495\textwidth]{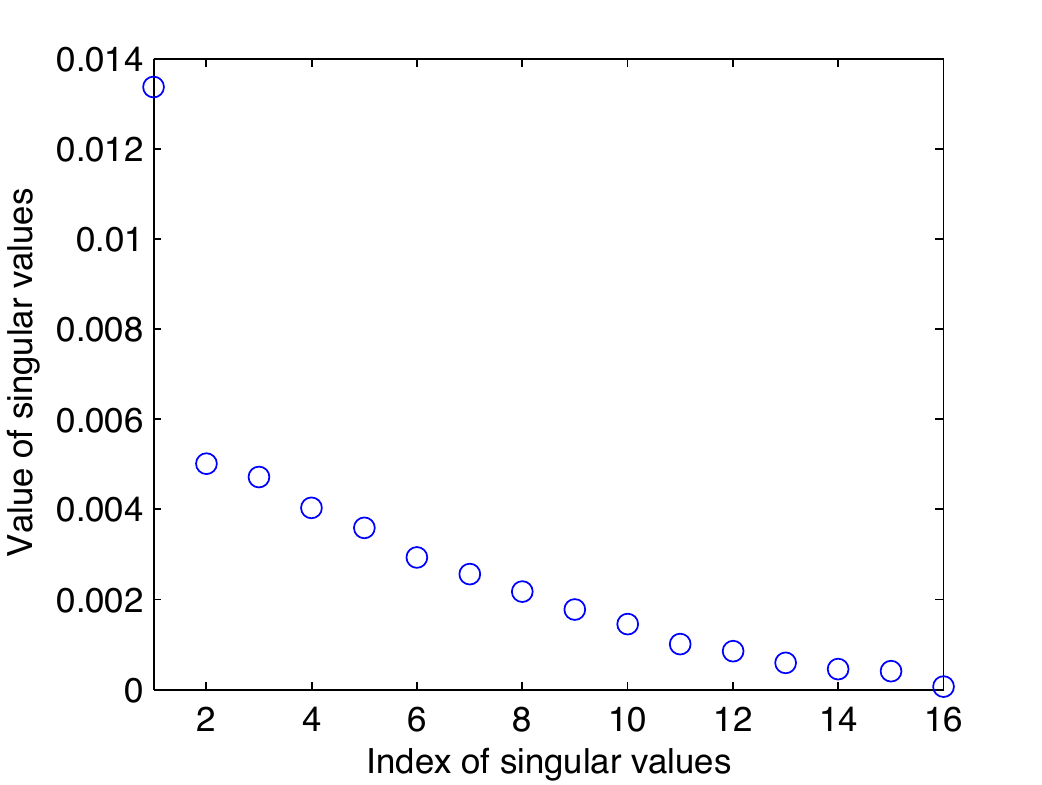}
\caption{\label{Distribution}(Section \ref{Section-Compare}) Distribution of singular values with (left) and without (right) diagonal elements of $\mathbb{F}$ at $f=1\mathrm{GHz}$.}
\end{center}
\end{figure}

\begin{figure}[h]
\begin{center}
\includegraphics[width=0.495\textwidth]{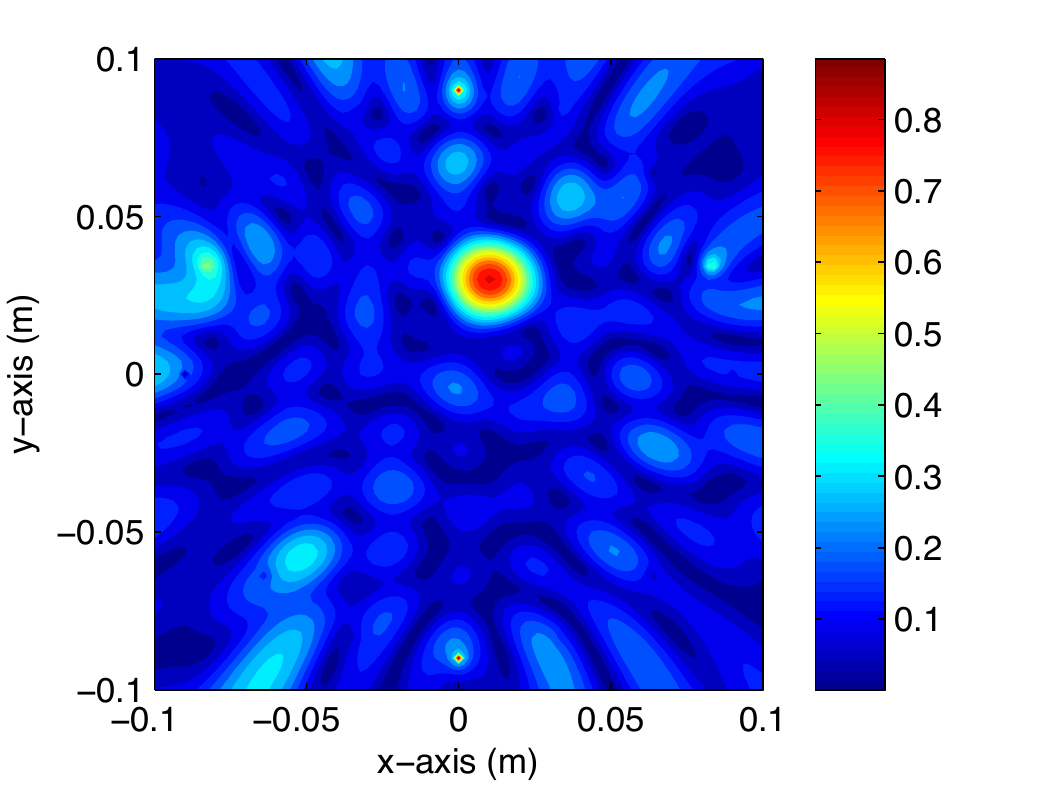}
\includegraphics[width=0.495\textwidth]{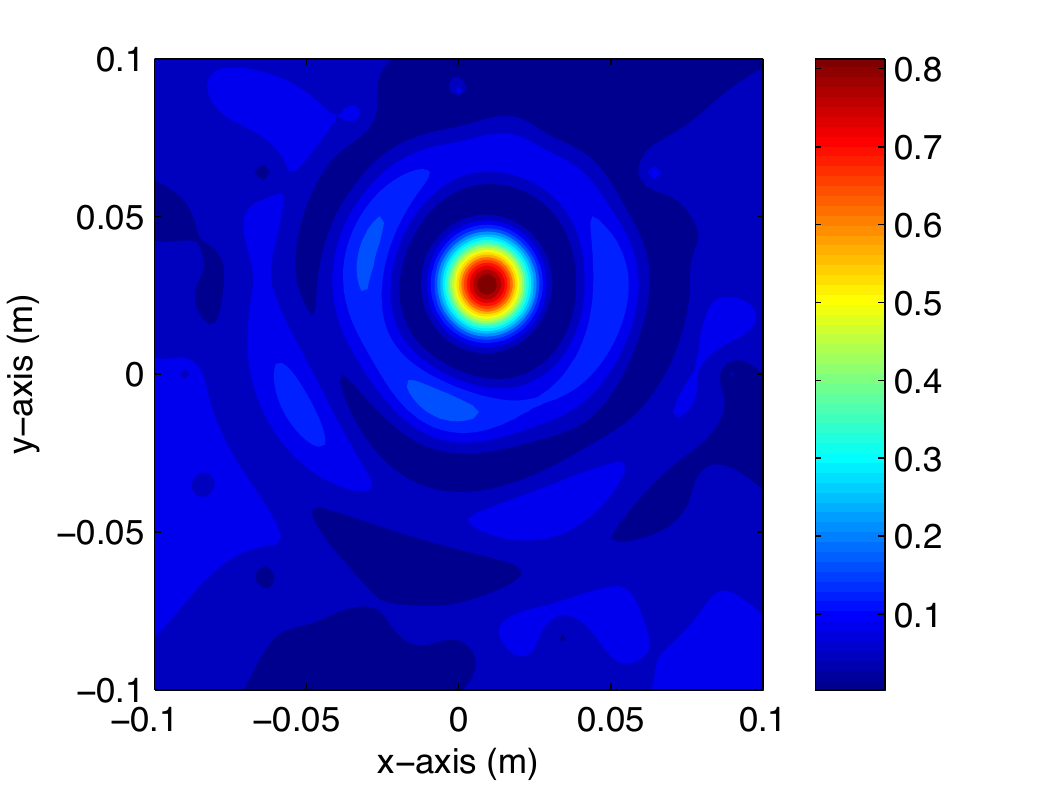}
\caption{\label{Result}(Section \ref{Section-Compare}) Maps of $\mathfrak{F}_{\mathrm{Full}}(\mr)$ (left column) and $\mathfrak{F}_{\mathrm{Diag}}(\mr)$ (right column) at $f=1\mathrm{GHz}$.}
\end{center}
\end{figure}

\begin{figure}[h]
\begin{center}
\includegraphics[width=0.495\textwidth]{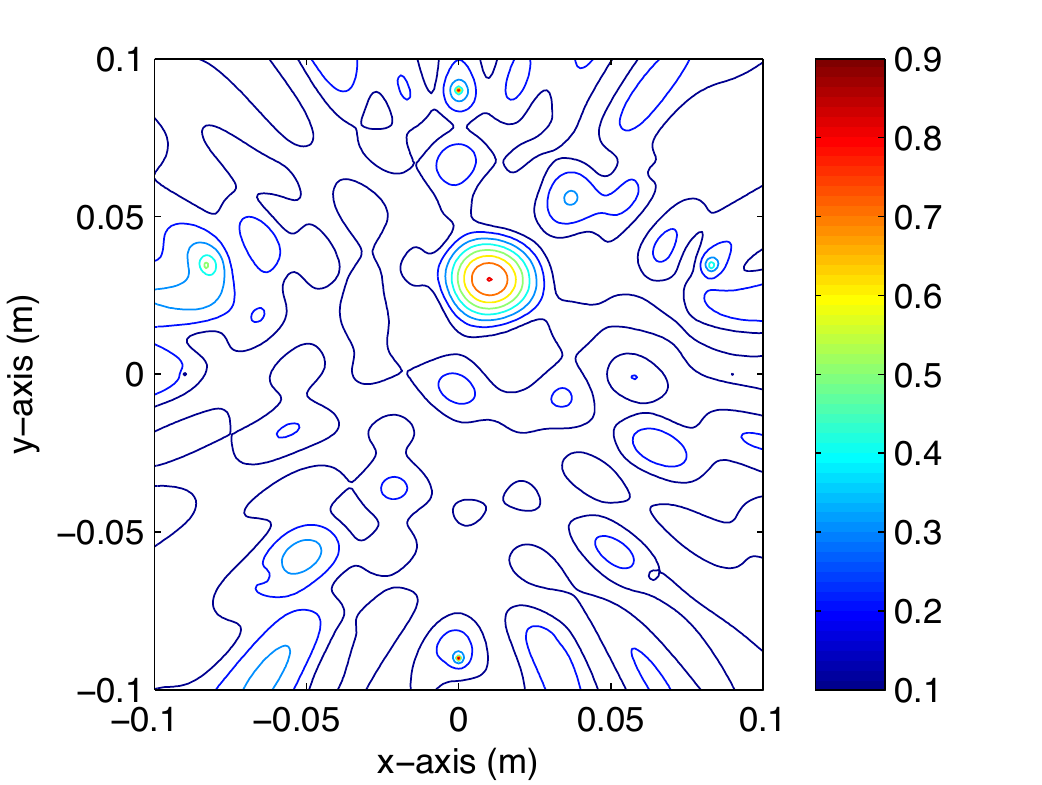}
\includegraphics[width=0.495\textwidth]{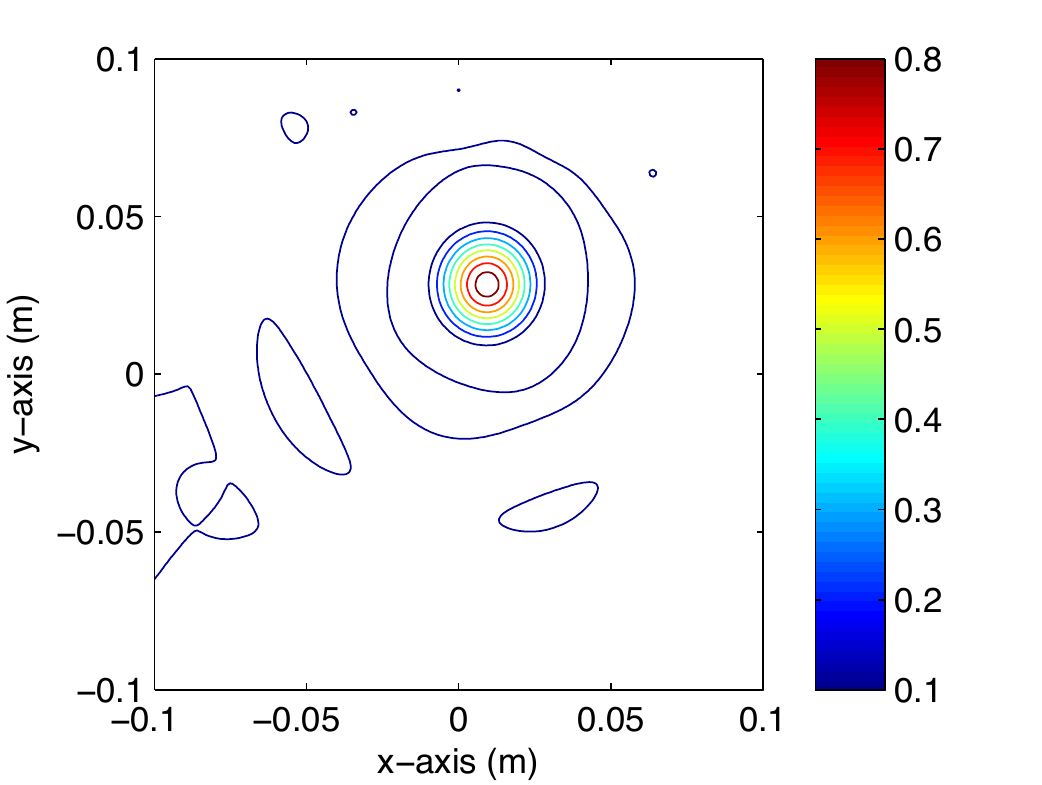}
\caption{\label{Contourplotsmall}(Section \ref{Section-Compare}) Contour lines of $\mathfrak{F}_{\mathrm{Full}}(\mr)$ (left column) and $\mathfrak{F}_{\mathrm{Diag}}(\mr)$ (right column) at $f=1\mathrm{GHz}$.}
\end{center}
\end{figure}

\subsection{Dependency of the narrowband dipole antenna}\label{Section-Antenna}
Now, let us consider the influence of the applied value of frequency. Maps of $\mathfrak{F}_{\mathrm{Diag}}(\mr)$ at several frequencies are shown in Figure \ref{Result-Frequencies-Good}. Notice that when $k$ is small, only a few artifacts appear; however, since the values of $J_s$, $s=1,2,\cdots,$ are not small, peaks of large magnitude will be contained within the artifacts. By contrast, when we apply high frequency, the shape/location of $\Sigma$ appears very clearly and accurately, but more artifacts with small magnitudes will also be included in the map of $\mathfrak{F}_{\mathrm{Diag}}(\mr)$.

Generally, it is very hard to produce broadband antenna. In this simulation, we used narrowband dipole antennas that are suitable in the frequency range $\mathcal{R}=\set{f^\star:S_{\mathrm{inc}}(1,1)\leq-10\mathrm{dB}}=\set{f^\star:0.75\leq f^\star\leq1.32\mathrm{ GHz}}$ (refer to Figure \ref{DipoleAntenna}). Figure \ref{Result-Frequencies-Bad} shows maps of $\mathfrak{F}_{\mathrm{Diag}}(\mr)$ at $f=0.5$ and $2.4$ $\mathrm{GHz}$. Notice that, at $f=0.5$ $\mathrm{GHz}$, the value of $\min\set{|f-f^\star|:f^\star\in\mathcal{R}}=0.25$ can be regarded as small; it is possible to recognize the existence of an anomaly, but the identified location, $\mr_\star$, is shifted, and a point whose magnitude is not so small also appears. Thus, it is very hard to obtain information about an anomaly. Notice that, at $f=2.4$ $\mathrm{GHz}$, the value of $\min\set{|f-f^\star|:f^\star\in\mathcal{R}}=1.08$ is not small, and dipole antennas are not useful for imaging, making the obtained result very poor.

\begin{figure}[h]
\begin{center}
\includegraphics[width=0.495\textwidth]{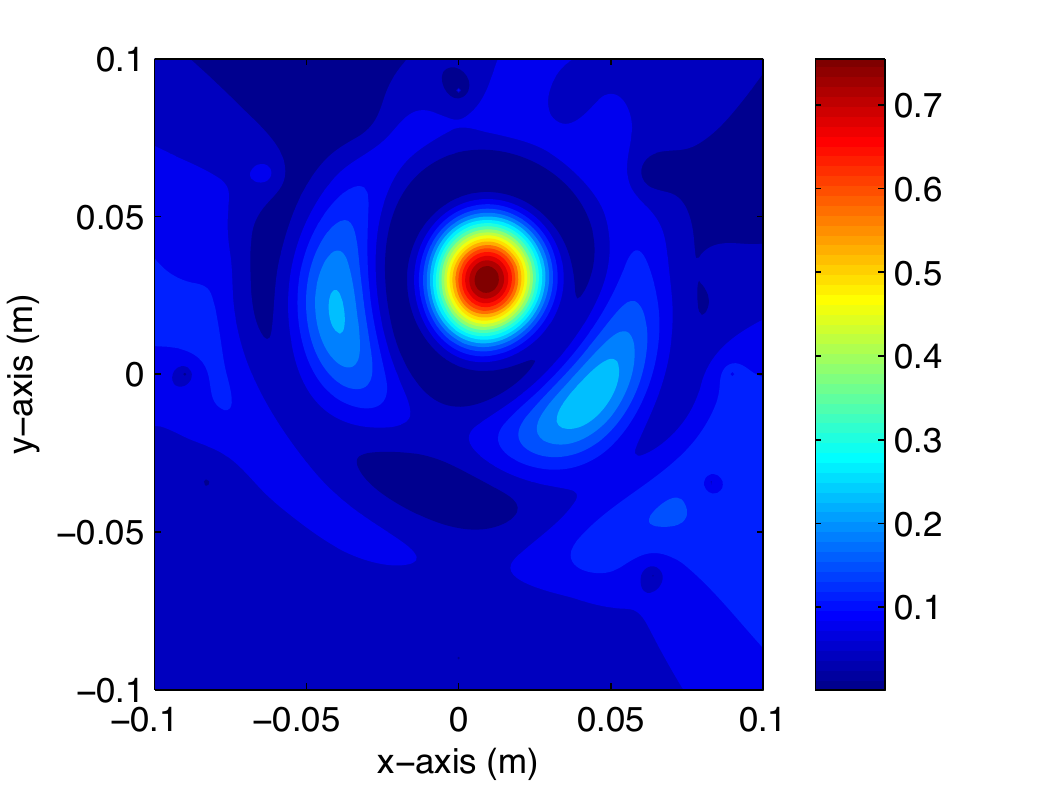}
\includegraphics[width=0.495\textwidth]{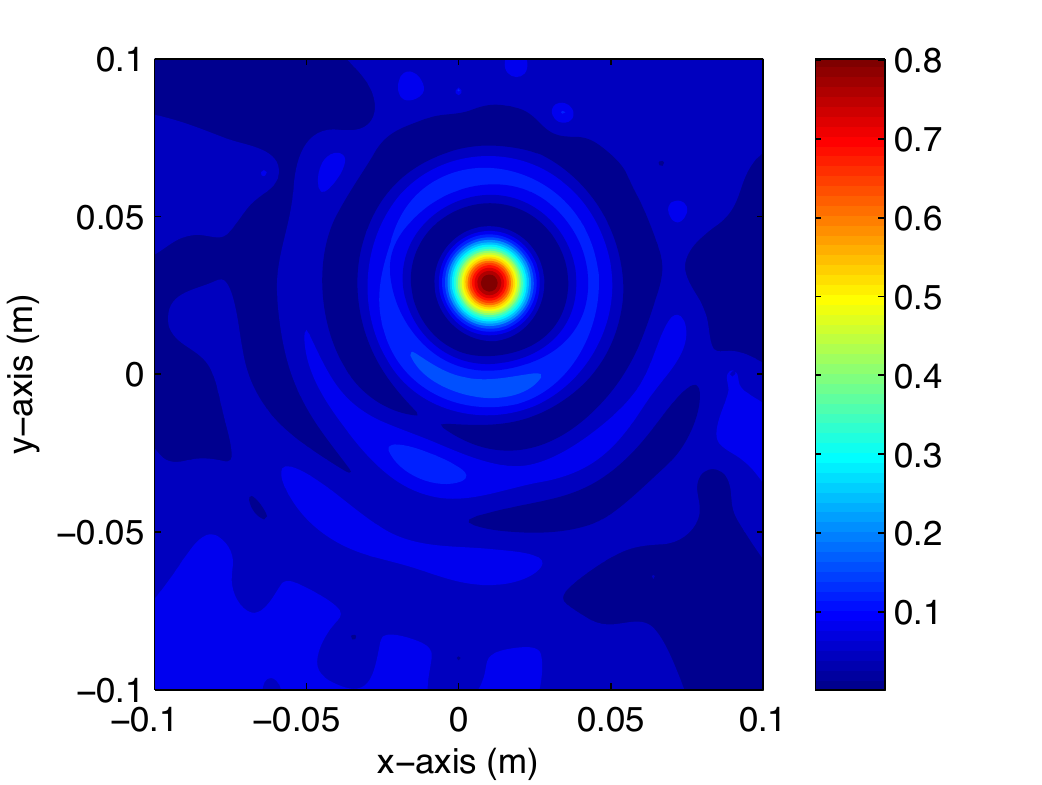}
\caption{\label{Result-Frequencies-Good}(Section \ref{Section-Antenna}) Maps of $\mathfrak{F}_{\mathrm{Diag}}(\mr)$ at $f=0.8$ (left) and $f=1.2$ (right) $\mathrm{GHz}$.}
\end{center}
\end{figure}

\begin{figure}[h]
\begin{center}
\includegraphics[width=0.9\textwidth]{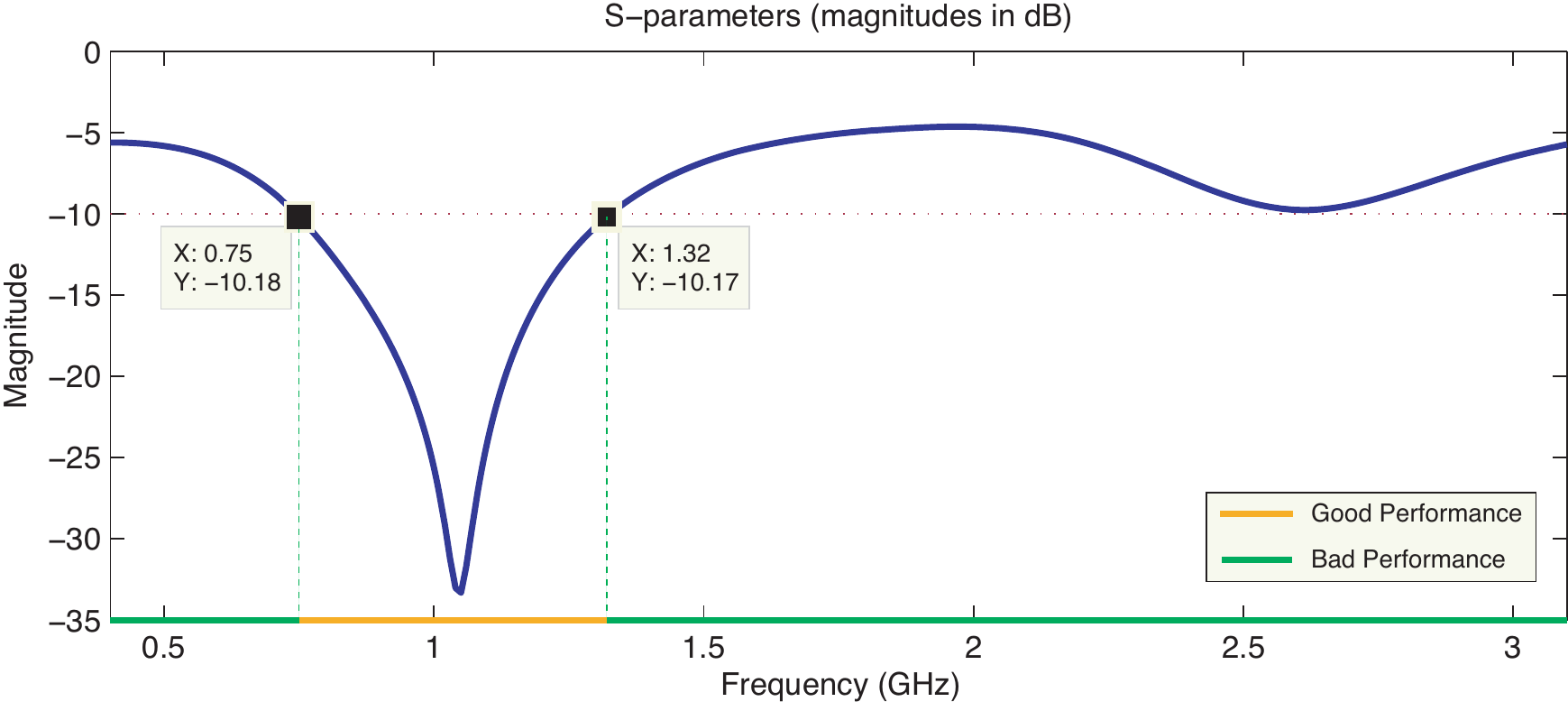}
\caption{\label{DipoleAntenna}(Section \ref{Section-Antenna}) Plot of magnitudes in dB of $S_{\mathrm{inc}}(1,1)$ from $f=0.4$ to $3.1$ $\mathrm{GHz}$.}
\end{center}
\end{figure}

\begin{figure}[h]
\begin{center}
\includegraphics[width=0.495\textwidth]{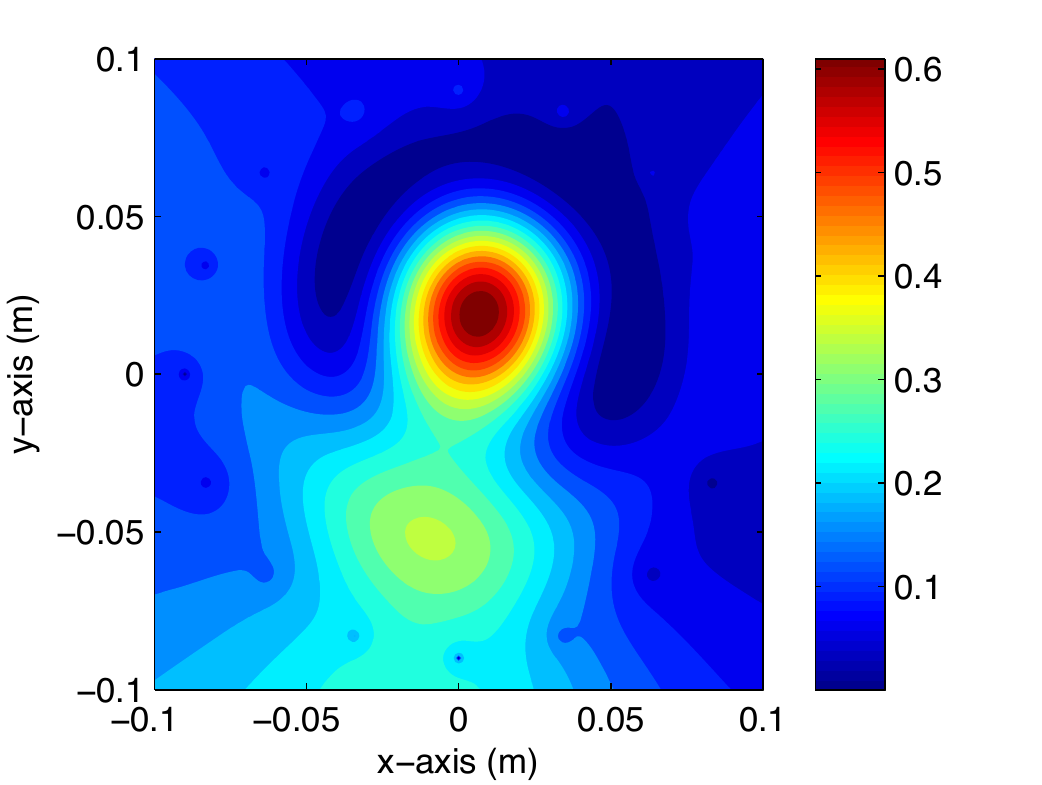}
\includegraphics[width=0.495\textwidth]{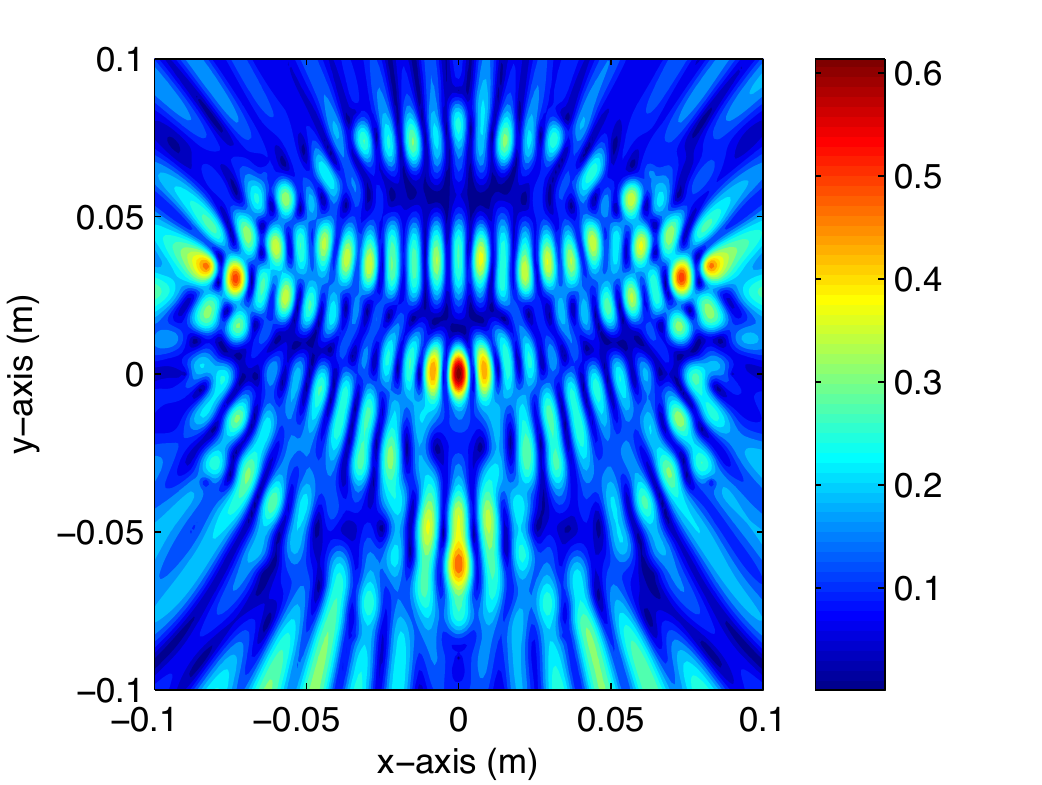}
\caption{\label{Result-Frequencies-Bad}(Section \ref{Section-Antenna}) Maps of $\mathfrak{F}_{\mathrm{Diag}}(\mr)$ at $f=0.5$ (left) and $f=2.4$ (right) $\mathrm{GHz}$.}
\end{center}
\end{figure}

\subsection{Imaging of an extended anomaly}\label{Section-Extended}
Now let us apply the technique developed for imaging an extended anomaly. Test configurations with and without anomalies are shown in Figure \ref{Configuration-Extended}. Figure \ref{Distribution-Extended} shows the distribution of singular values with and without diagonal elements of $\mathbb{F}$. Notice that, unlike imaging of a small anomaly, it is very hard to discriminate nonzero singular values of $\mathbb{F}$ with diagonal elements. Fortunately, it is easy to do so when we eliminate the diagonal elements of $\mathbb{F}$.

On the basis of the maps of $\mathfrak{F}_{\mathrm{Full}}(\mr)$ and $\mathfrak{F}_{\mathrm{Diag}}(\mr)$ in Figure \ref{Result-Extended}, we can see that, although it is impossible to obtain the exact shape of an anomaly, an outline of this shape can be recognized via the map of $\mathfrak{F}_{\mathrm{Full}}(\mr)$. It is interesting to observe that the result obtained via the map $\mathfrak{F}_{\mathrm{Diag}}(\mr)$ is much better than that of the $\mathfrak{F}_{\mathrm{Full}}(\mr)$. We also refer to Figure \ref{Contourplotextended} to compare the artifacts. Hence, we can conclude that proposed method is an improvement over the traditional approach and that by taking an outline as an initial guess, the complete shape of an anomaly can be retrieved by applying a Newton-type method, level-set technique, or other quantitative inversion
strategies, refer to \cite{KSY,DL,S1,PL3,ADIM,BK1,CW,PBCID}.

\begin{figure}[h]
\begin{center}
\includegraphics[width=0.495\textwidth]{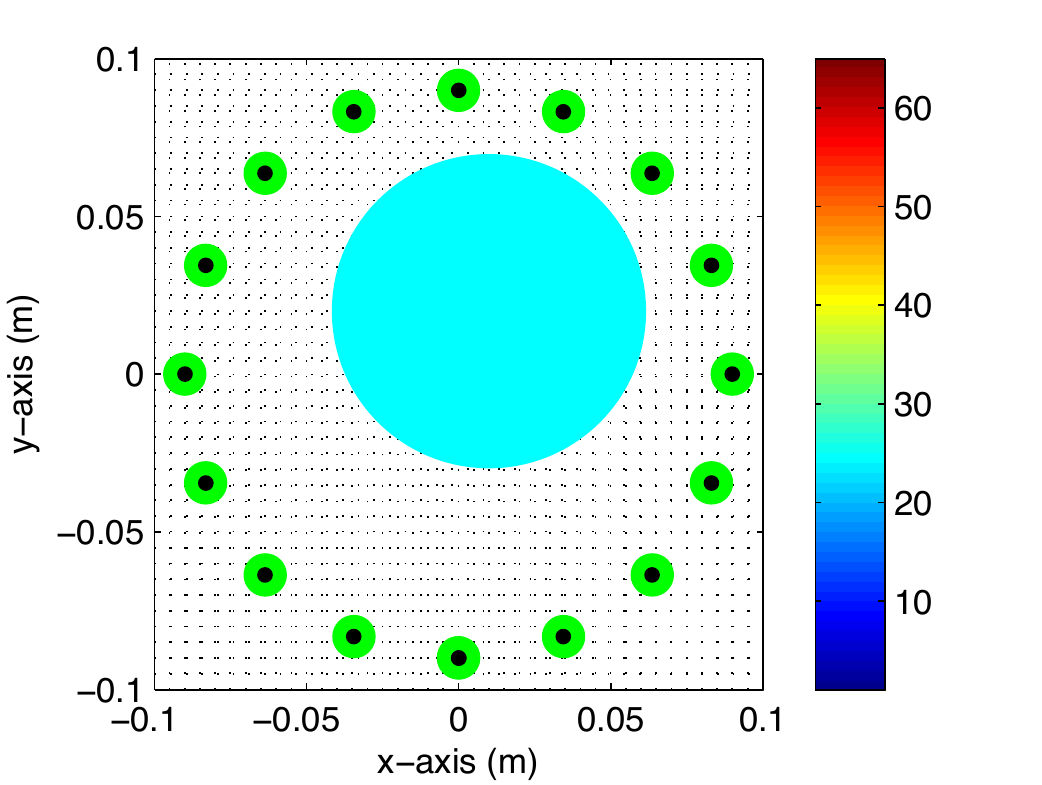}
\includegraphics[width=0.495\textwidth]{Configuration-homo.pdf}
\caption{\label{Configuration-Extended}(Section \ref{Section-Extended}) Test configurations with (left) and without (right) anomalies.}
\end{center}
\end{figure}

\begin{figure}[h]
\begin{center}
\includegraphics[width=0.495\textwidth]{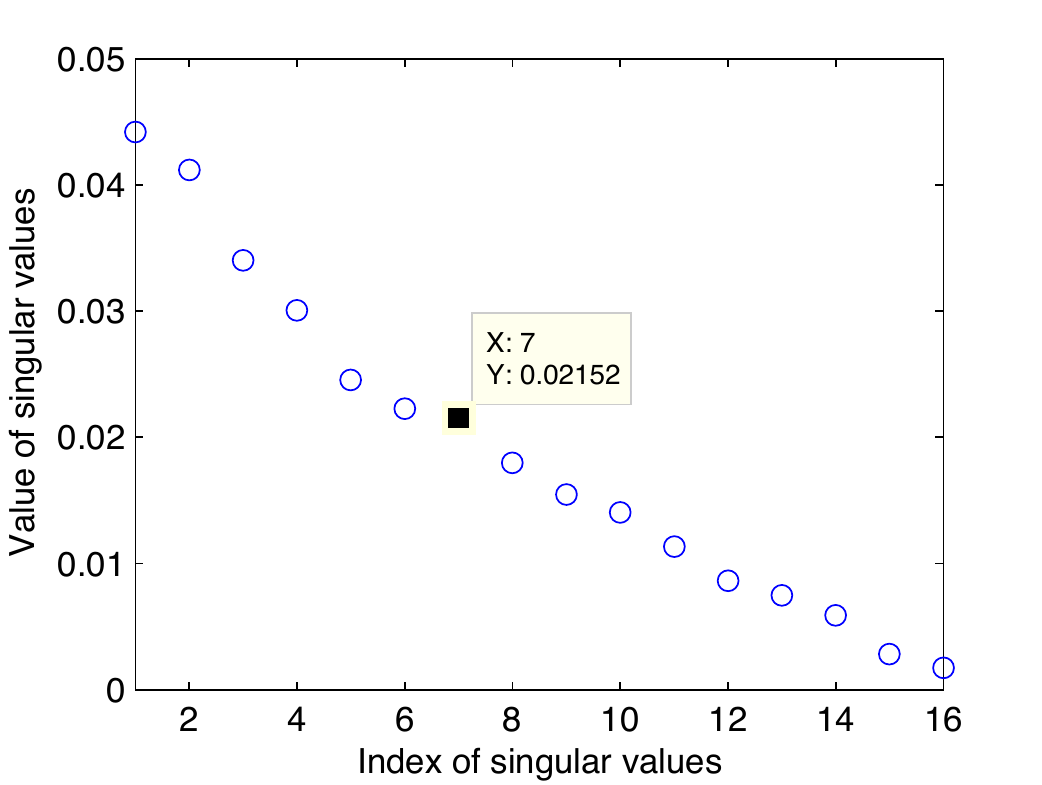}
\includegraphics[width=0.495\textwidth]{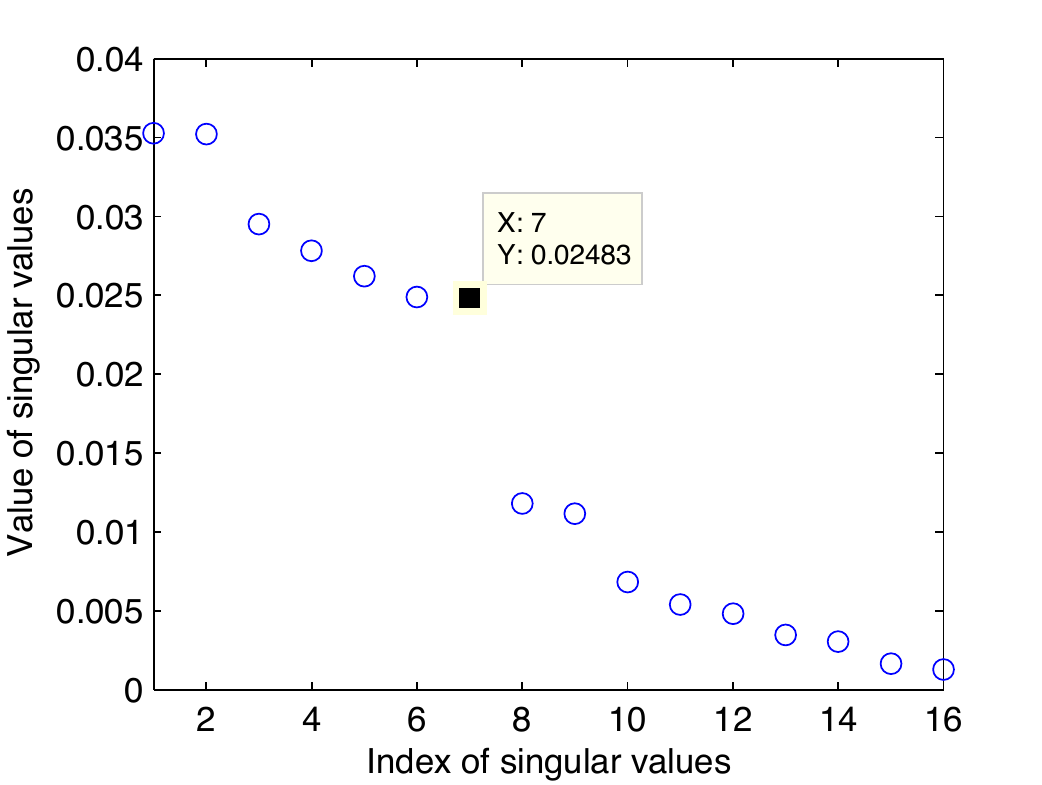}
\caption{\label{Distribution-Extended}(Section \ref{Section-Extended}) Distribution of singular values with (left) and without (right) diagonal elements of $\mathbb{F}$ at $f=1\mathrm{GHz}$.}
\end{center}
\end{figure}

\begin{figure}[h]
\begin{center}
\includegraphics[width=0.495\textwidth]{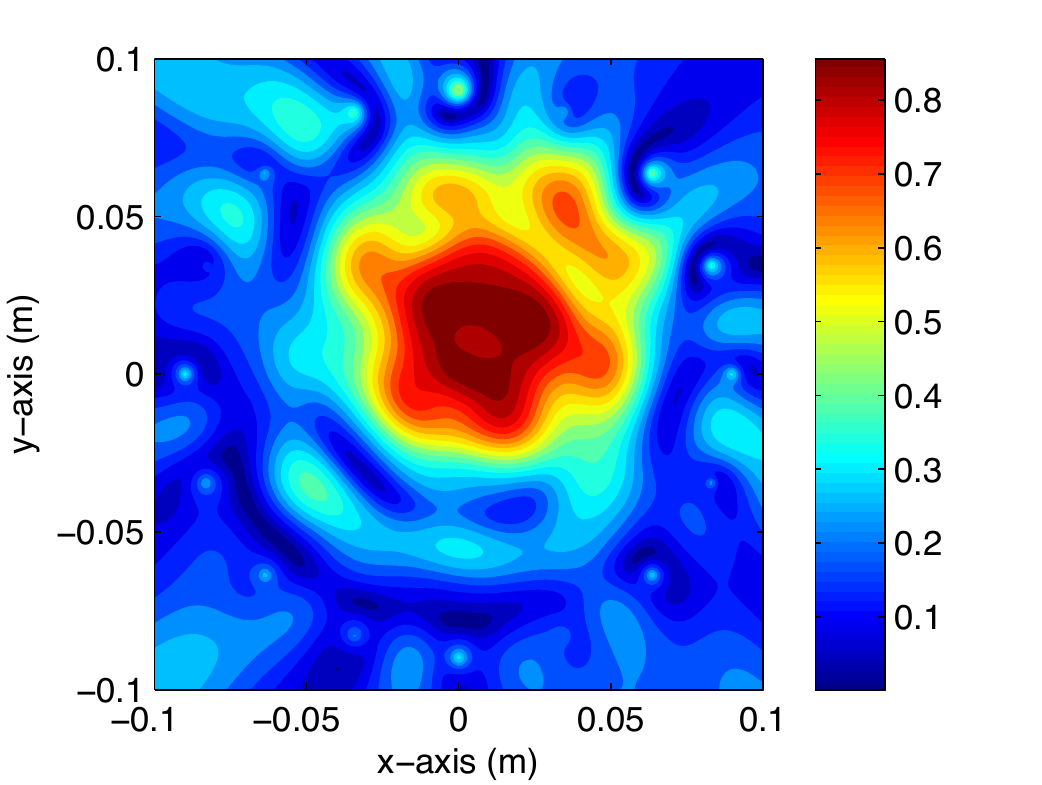}
\includegraphics[width=0.495\textwidth]{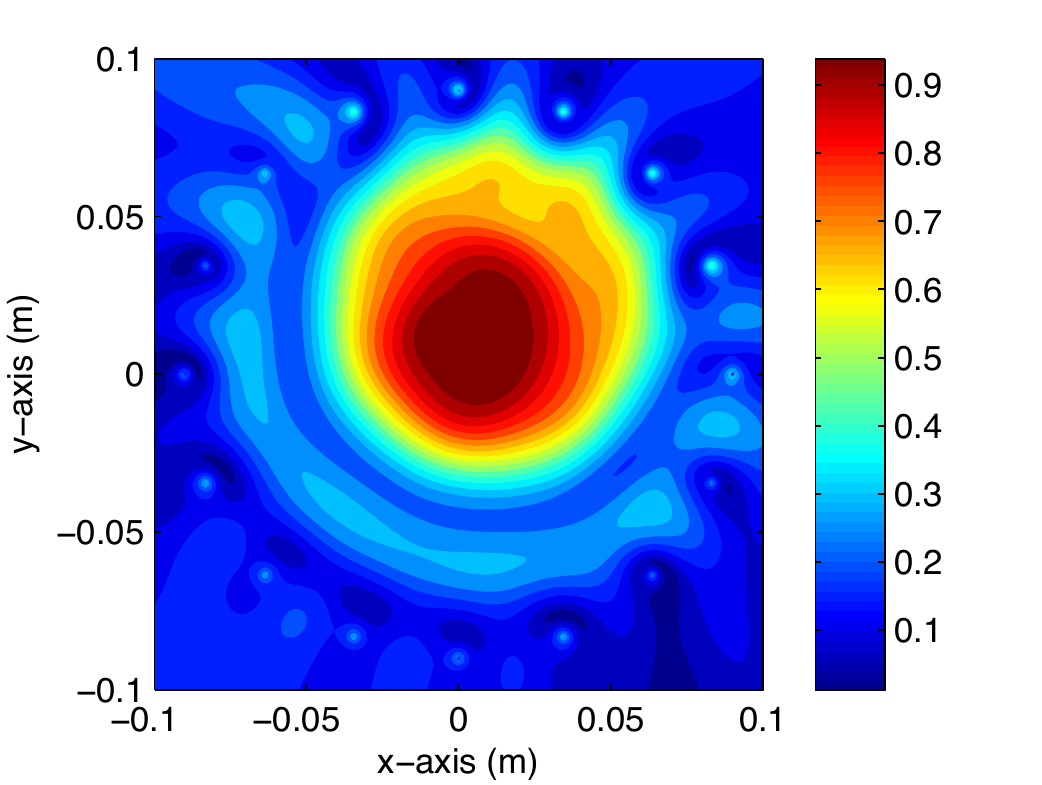}
\caption{\label{Result-Extended}(Section \ref{Section-Extended}) Maps of $\mathfrak{F}_{\mathrm{Full}}(\mr)$ (left column) and $\mathfrak{F}_{\mathrm{Diag}}(\mr)$ (right column) at $f=1\mathrm{GHz}$.}
\end{center}
\end{figure}

\begin{figure}[h]
\begin{center}
\includegraphics[width=0.495\textwidth]{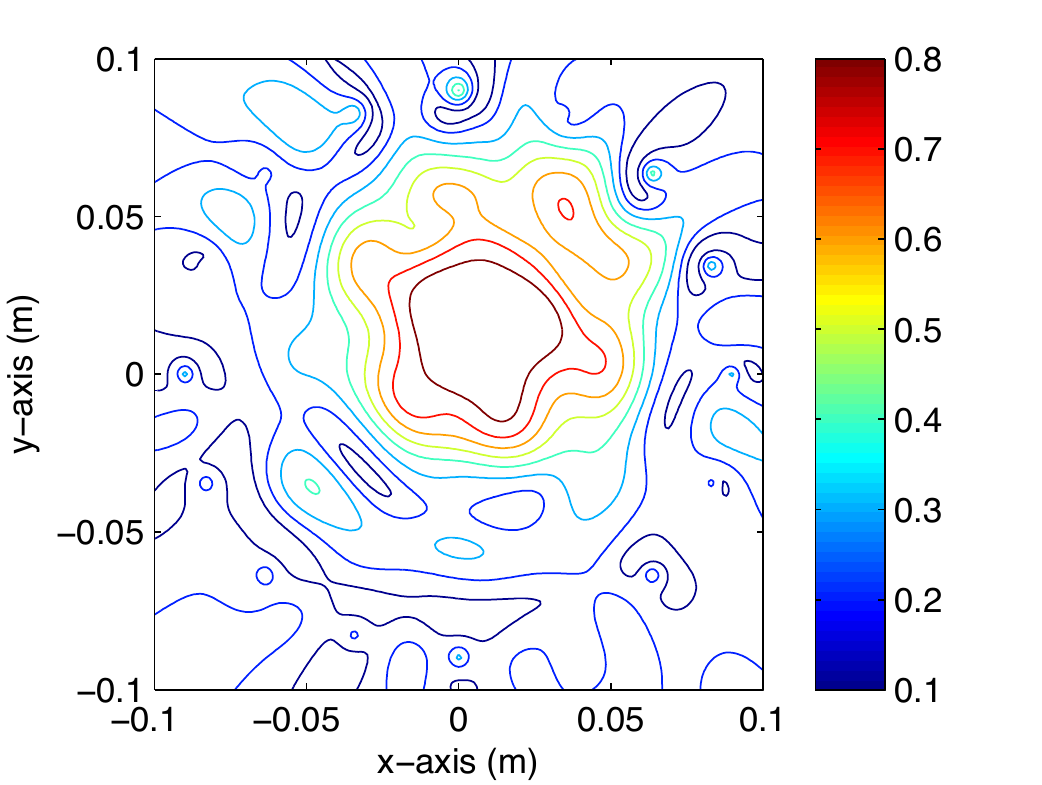}
\includegraphics[width=0.495\textwidth]{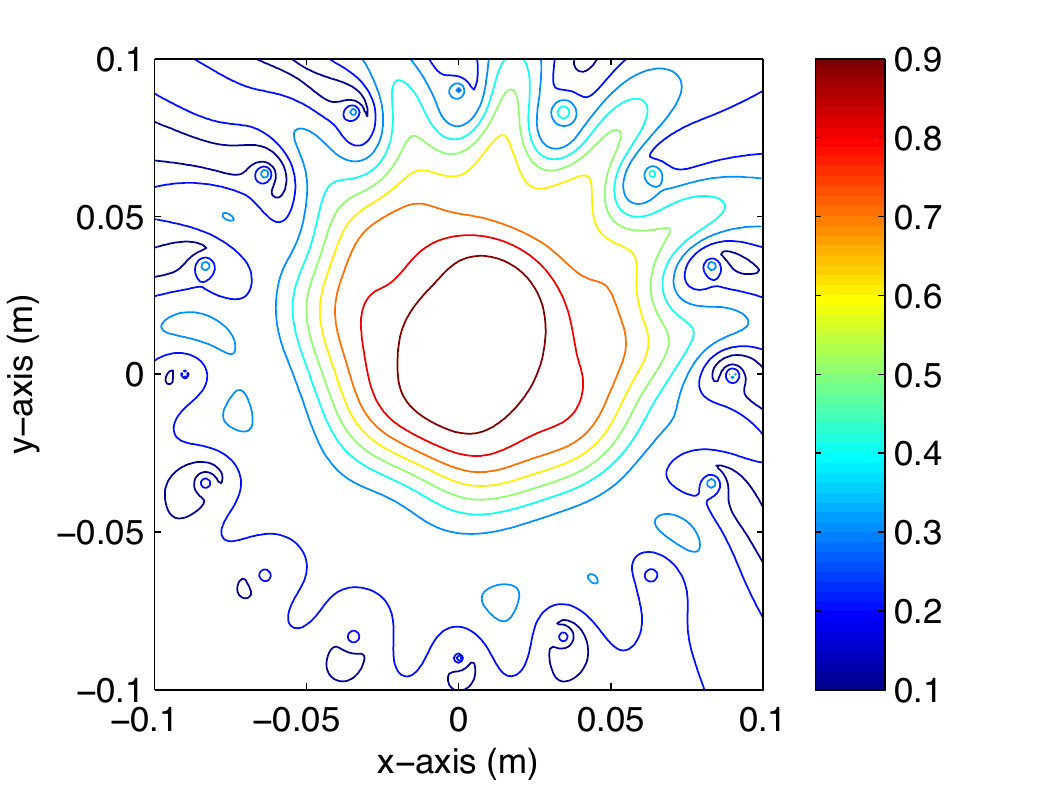}
\caption{\label{Contourplotextended}(Section \ref{Section-Extended}) Contour lines of $\mathfrak{F}_{\mathrm{Full}}(\mr)$ (left column) and $\mathfrak{F}_{\mathrm{Diag}}(\mr)$ (right column) at $f=1\mathrm{GHz}$.}
\end{center}
\end{figure}

\subsection{Imaging of small anomalies: real-data experiment}\label{Section-Real}
Here, we perform real experiments to demonstrate the viability of the designed algorithm. The microwave machine manufactured by the research team of the Radio Environment \& Monitoring research group of the Electronics and Telecommunications Research Institute (ETRI) is illustrated in Figure \ref{RealdataSimulation}. For the simulation, we filled this machine with water with permittivity $78$ and conductivity $0.2$S/m, and search domain is selected as a circle of radius $0.085$m. The cross-section of one screwdriver to illustrate a single anomaly and two screw drivers and a hand hammer to simulate multiple anomalies (Figure \ref{Configuration-Real}) are considered.

\begin{figure}[h]
\begin{center}
\includegraphics[width=0.495\textwidth]{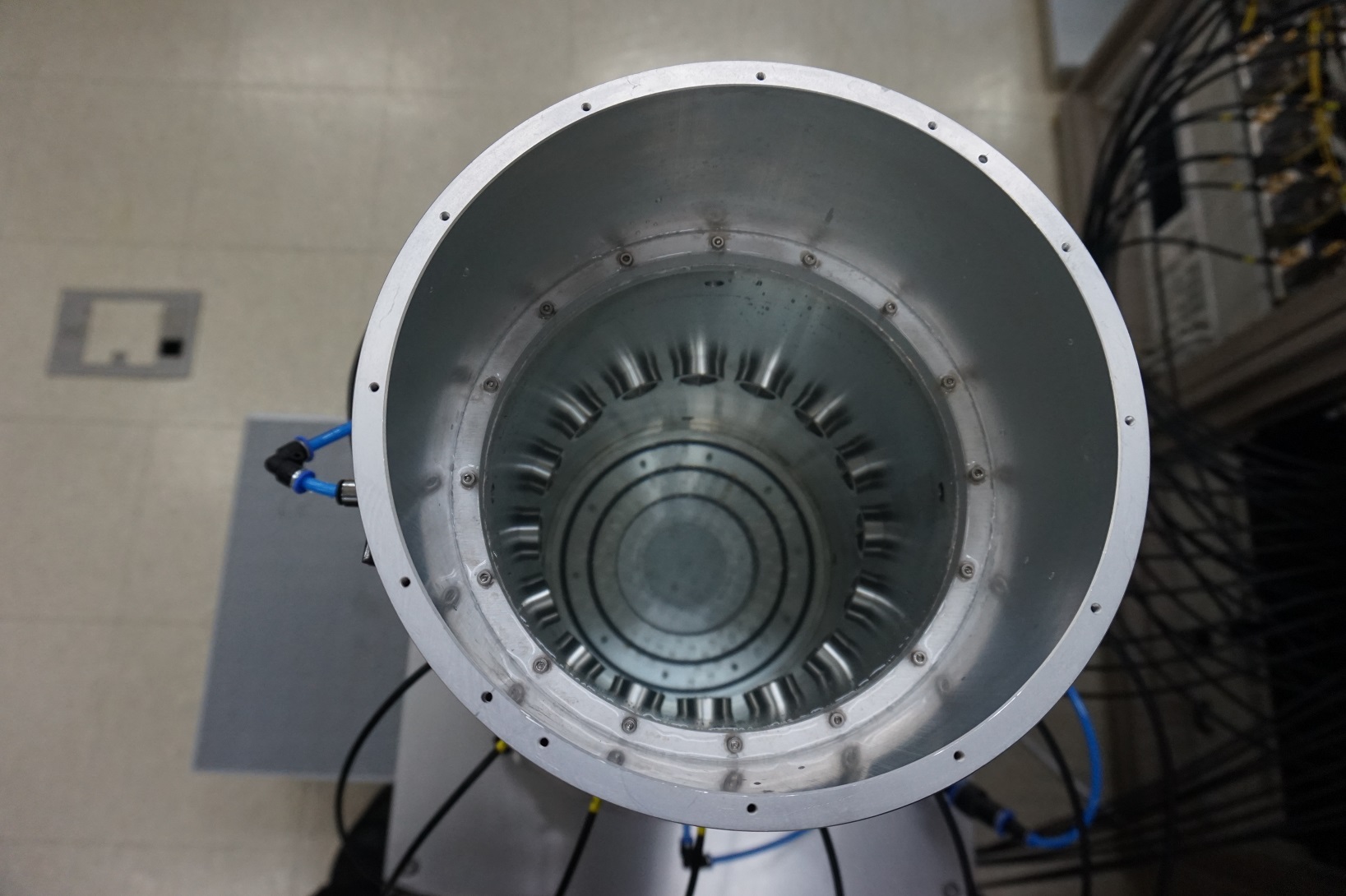}
\includegraphics[width=0.495\textwidth]{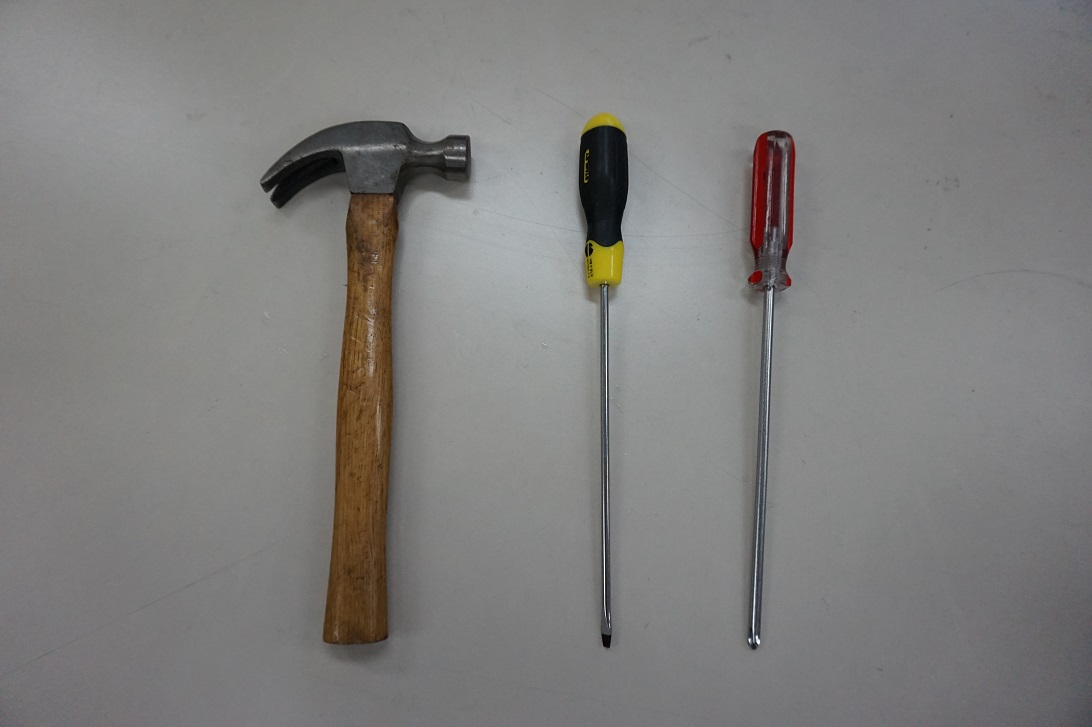}
\caption{\label{RealdataSimulation}(Section \ref{Section-Real}) Real microwave machine data (left) and materials used for anomalies (right).}
\end{center}
\end{figure}

\begin{figure}[h]
\begin{center}
\includegraphics[width=0.495\textwidth]{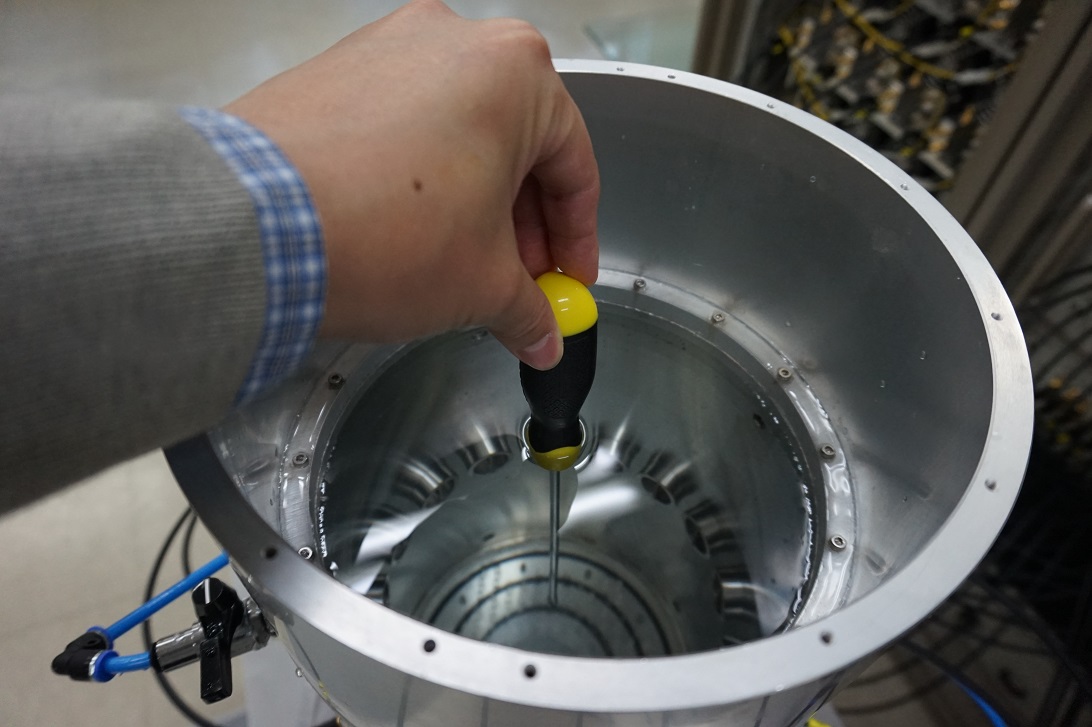}
\includegraphics[width=0.495\textwidth]{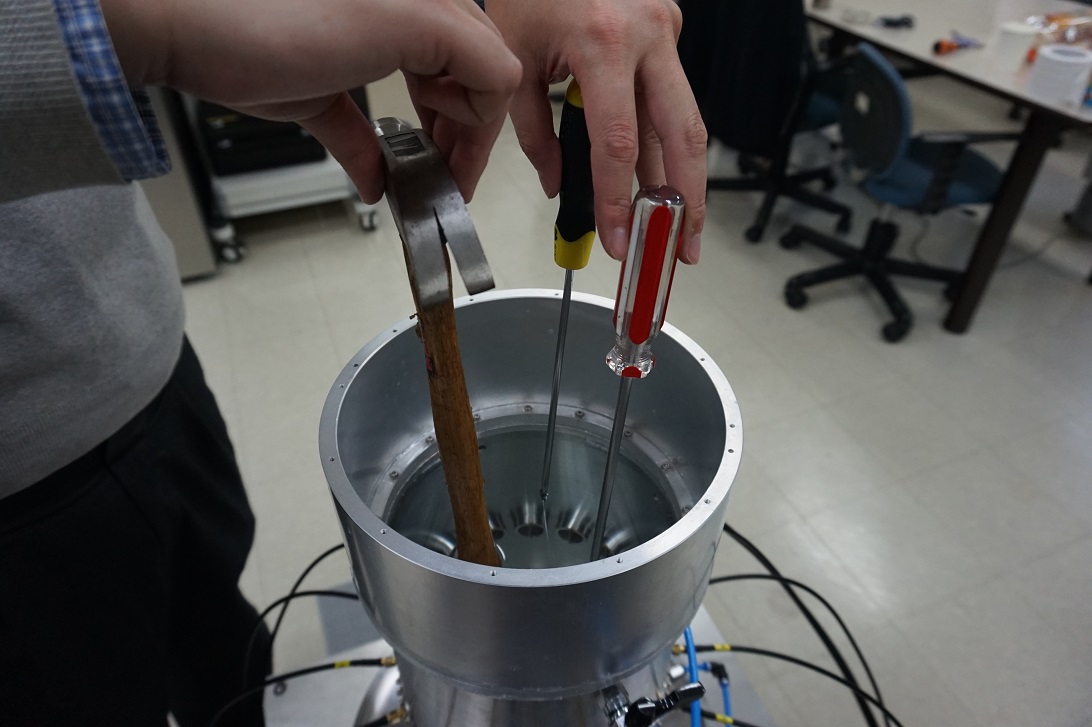}
\caption{\label{Configuration-Real}(Section \ref{Section-Real}) Description of real-data simulation for single (left) and multiple (right) anomalies.}
\end{center}
\end{figure}

Figure \ref{Result-Real} shows the description of simulations and map of $\mathfrak{F}_{\mathrm{Diag}}(\mr)$ at $f=925$ MHz for imaging single and multiple anomalies. Just as with the previous result in Figure \ref{Result}, the shape of the screwdriver is retrieved successfully. For multiple anomalies, although some artifacts appear, we can recognize the outline shape of three anomalies. Furthermore, it is possible to observe that one anomaly is larger than the others. This is due to the fact that the size of the hand hammer is larger than that of the screw drivers. However, unlike imaging of single anomalies, a careful threshold to discriminate nonzero singular values is necessary. In this result, we adopt $0.02-$threshold scheme (keeping only the values $\tau_n$ such that $\tau_n\geq0.02$).

\begin{figure}[h]
\begin{center}
\includegraphics[width=0.495\textwidth]{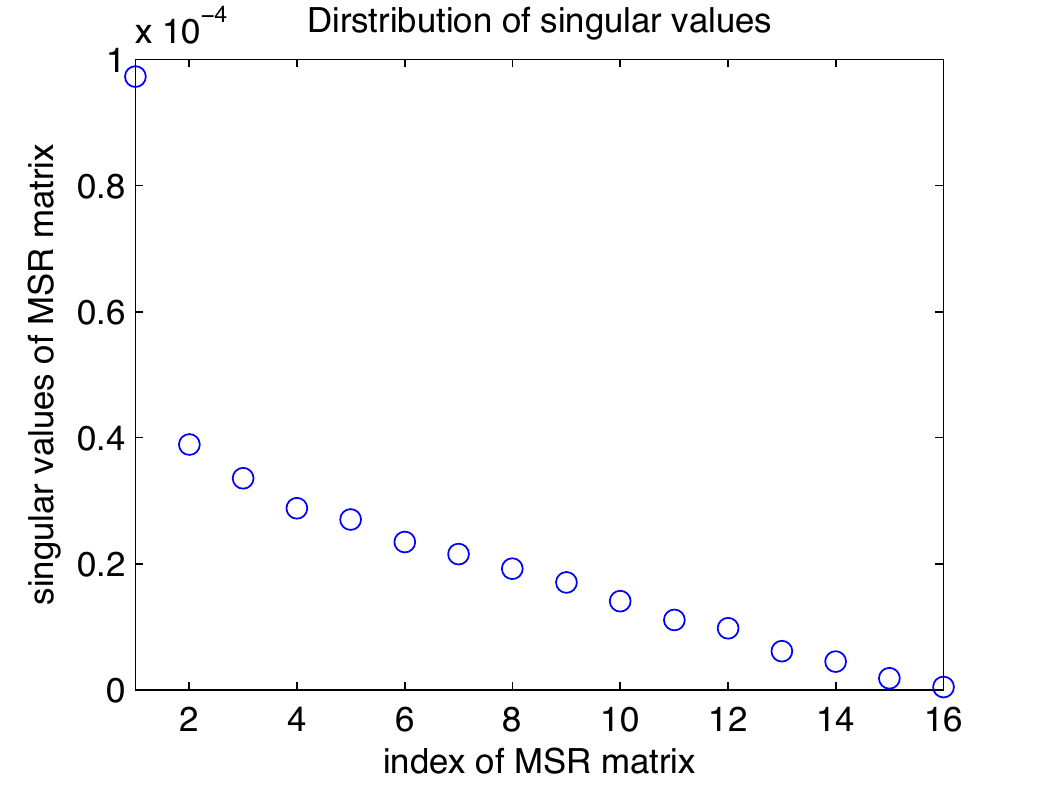}
\includegraphics[width=0.495\textwidth]{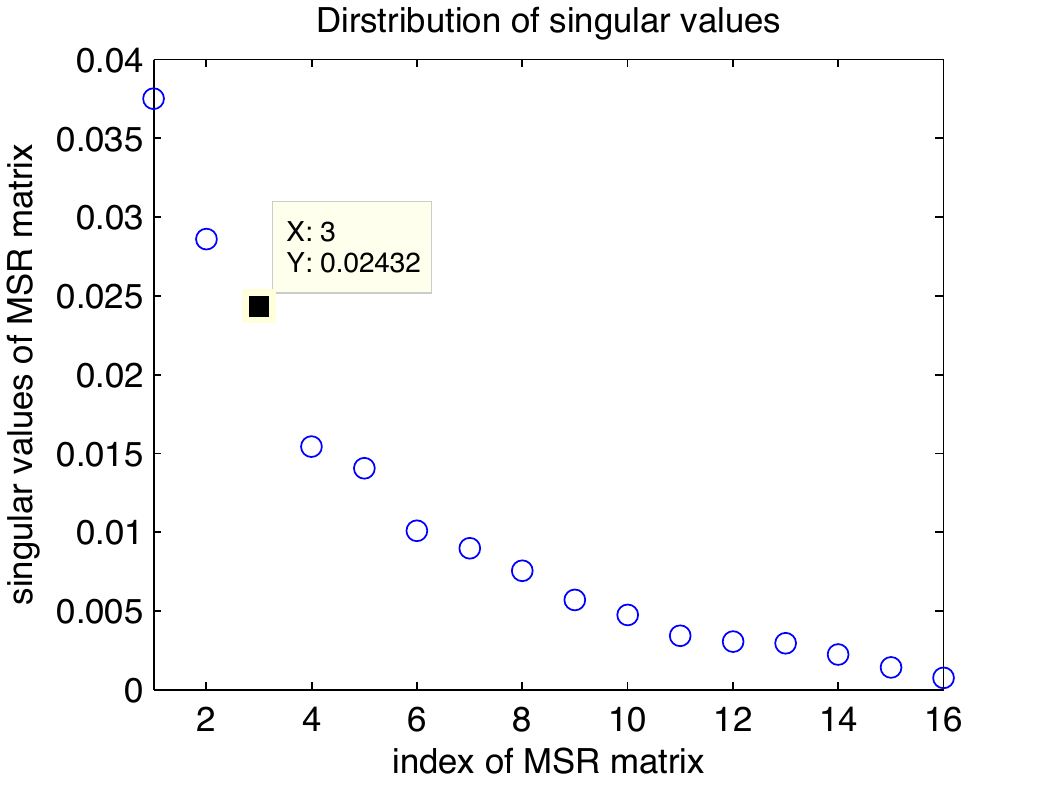}\\
\includegraphics[width=0.495\textwidth]{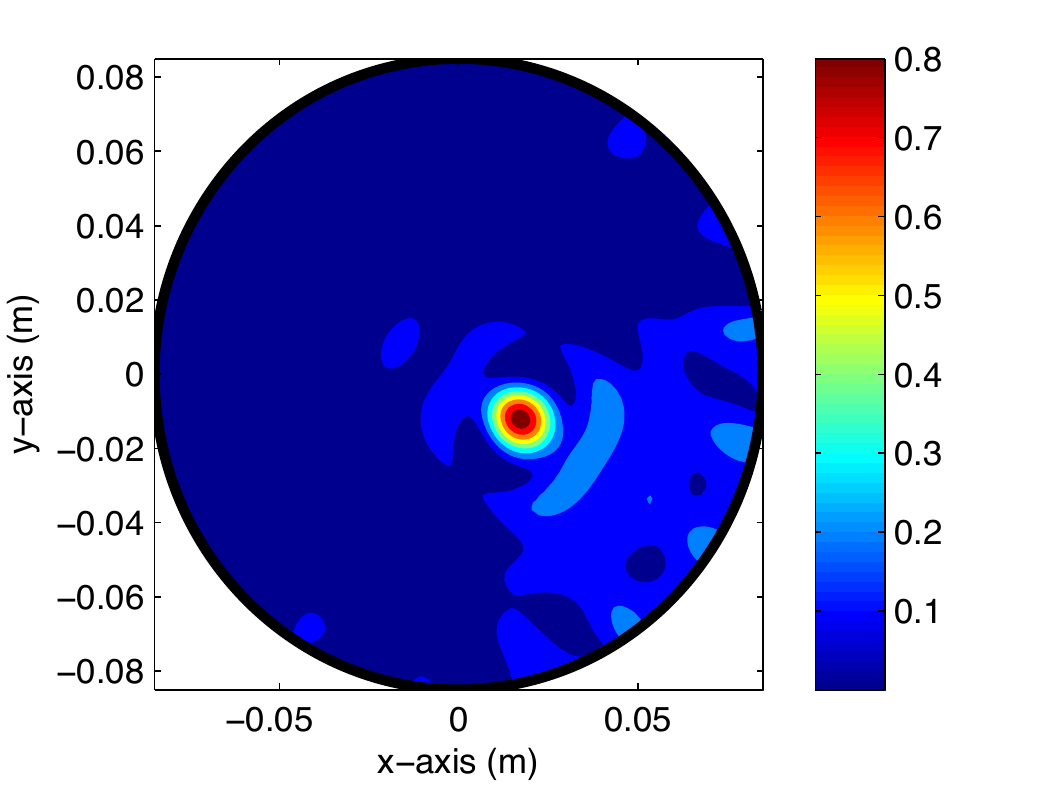}
\includegraphics[width=0.495\textwidth]{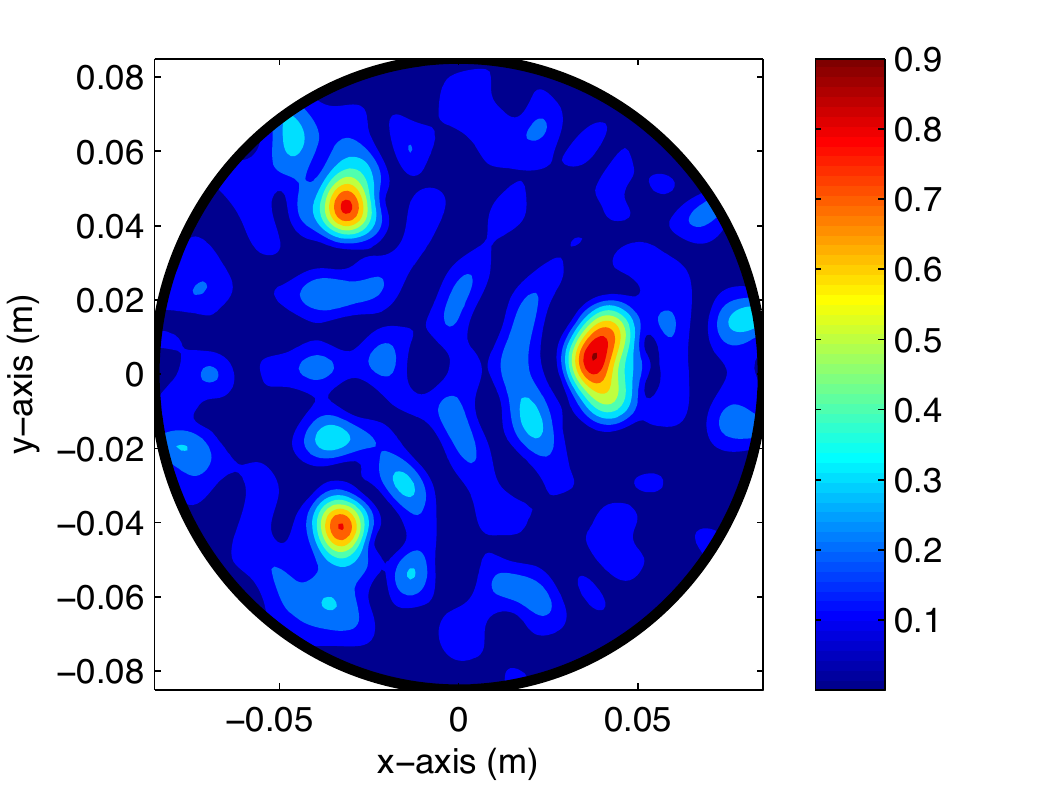}
\caption{\label{Result-Real}(Section \ref{Section-Real}) Top line: description of real-data simulations for single (left) and multiple (right) anomalies. Bottom line: maps of $\mathfrak{F}_{\mathrm{Diag}}(\mr)$ at $f=925~\mathrm{MHz}$ for single (left) and multiple (right) anomalies.}
\end{center}
\end{figure}

\subsection{Further results: limitations of the real-data experiment}\label{Section-Real-Limit}
Here, we discuss the limitations of the imaging technique. For illustration, we consider the cross-section imaging of one screwdriver and one plastic straw for multiple anomalies and an air-filled plastic bottle as an extended anomaly (refer to Figure \ref{Configuration-Real-Limit}).

The imaging result for multiple small anomalies of $\mathfrak{F}_{\mathrm{Diag}}(\mr)$ at $f=925$ MHz is shown in Figure \ref{Result-Real-Limit}. Notice that the values of permittivity of the plastic straw and screwdriver are extremely small and large, respectively; only one singular value that is significantly larger than the others appears. Correspondingly, it is possible to recognize the location of screwdriver in the map of $\mathfrak{F}_{\mathrm{Diag}}(\mr)$, but it is very hard to identify the location of plastic straw.

We next consider the imaging of an extended anomaly. To obtain the map of $\mathfrak{F}_{\mathrm{Diag}}(\mr)$ at $f=925$ MHz, the first eight singular values were used to define the imaging function; however, in general, we have observed that it is very hard to discriminate the nonzero singular value. Since the size of the anomaly is large, the Born approximation cannot be applied so that the simulation results do not match the theoretical results in Theorem \ref{TheoremStructure}. Although it was impossible to identify the complete shape of the extended anomaly, its outline shape was recognized. This is similar result as in \cite{BPV,HSZ1} about the imaging of an extended target. Hence, this result can be regarded as a good initial guess of Newton-type schemes or level-set strategies.

\begin{figure}[h]
\begin{center}
\includegraphics[width=0.495\textwidth]{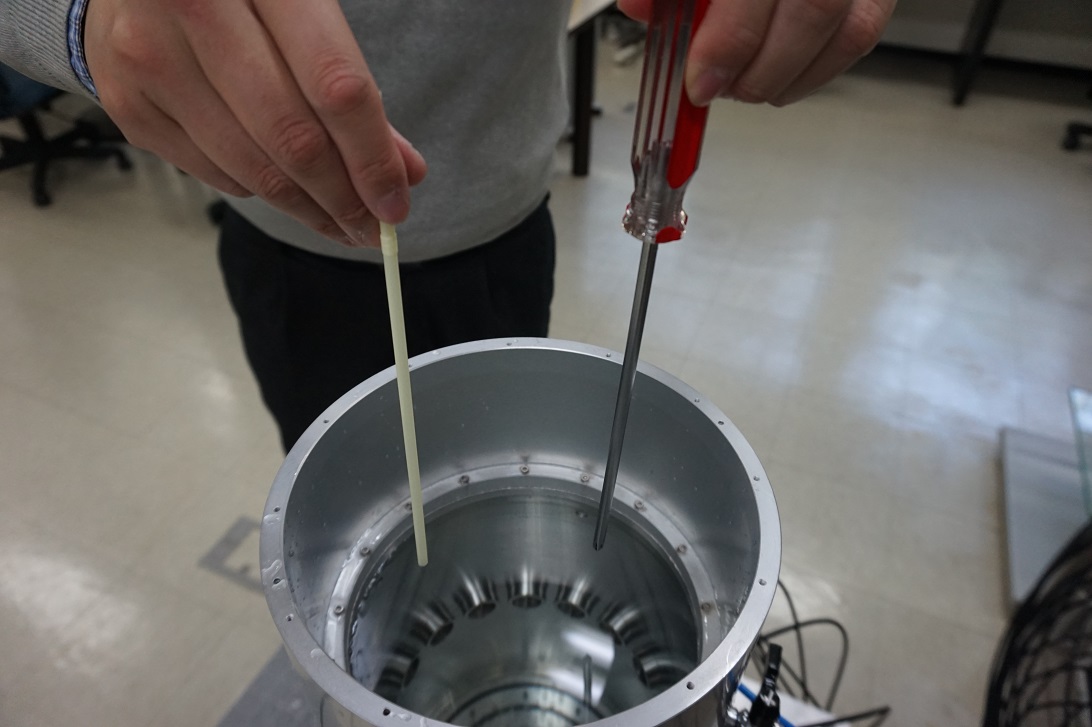}
\includegraphics[width=0.495\textwidth]{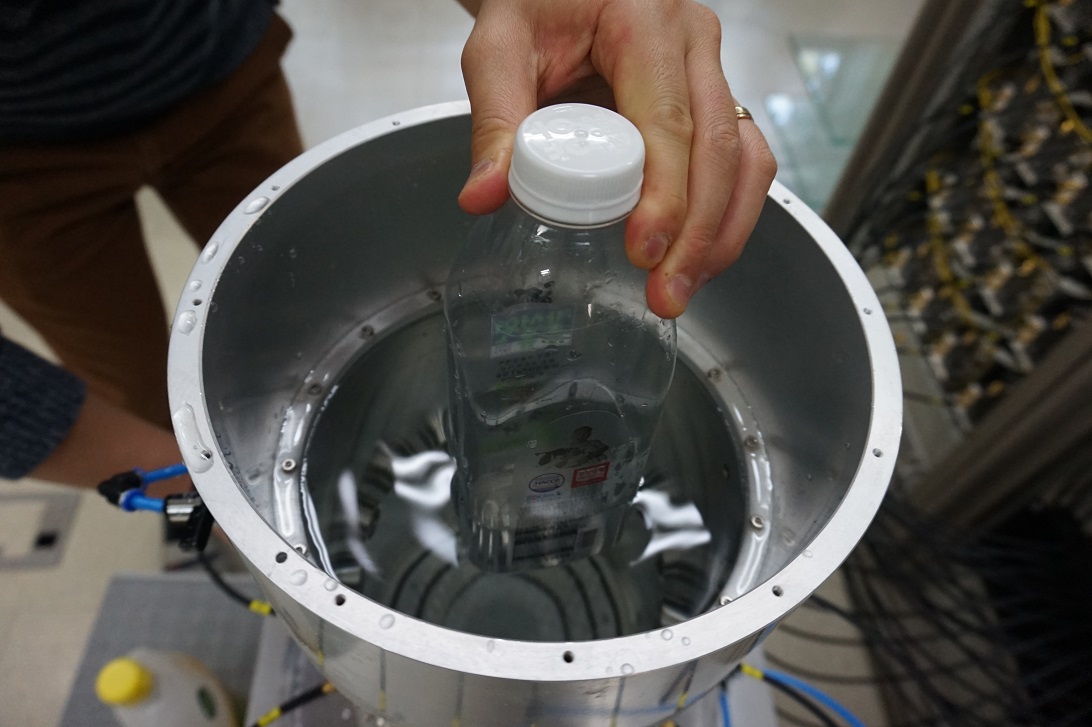}
\caption{\label{Configuration-Real-Limit}(Section \ref{Section-Real-Limit}) Description of real-data simulation for multiple small anomalies (left) and a rectangular extended anomaly (right).}
\end{center}
\end{figure}

\begin{figure}[h]
\begin{center}
\includegraphics[width=0.495\textwidth]{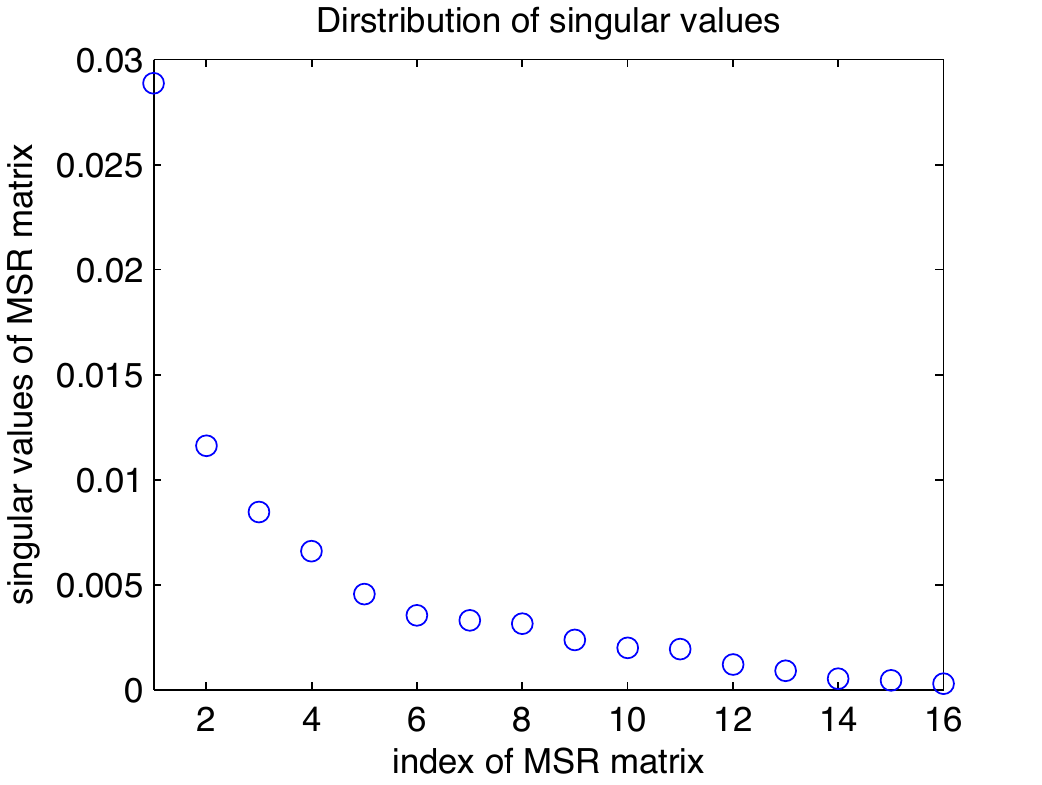}
\includegraphics[width=0.495\textwidth]{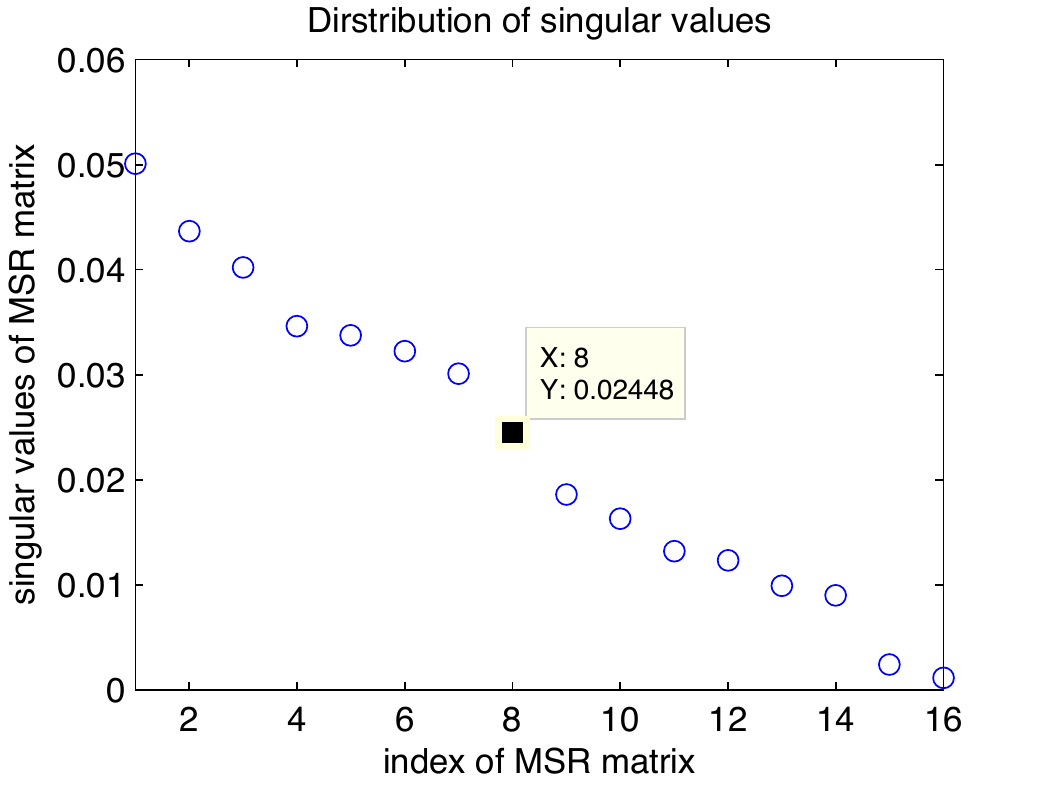}\\
\includegraphics[width=0.495\textwidth]{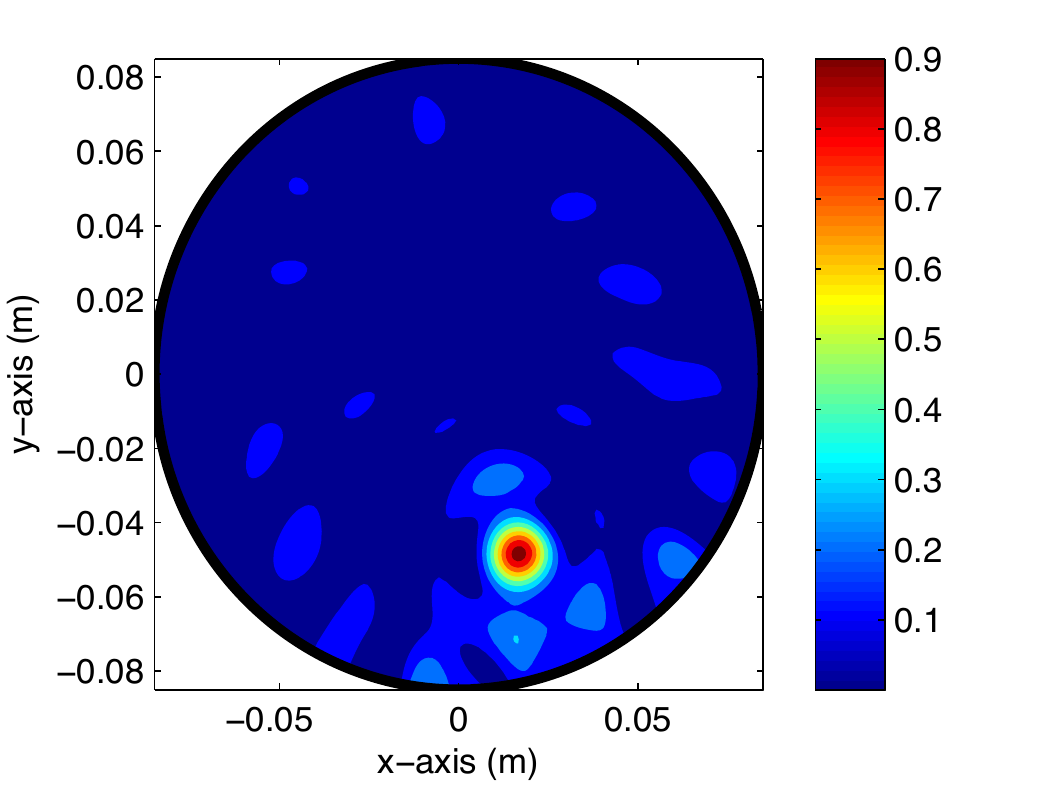}
\includegraphics[width=0.495\textwidth]{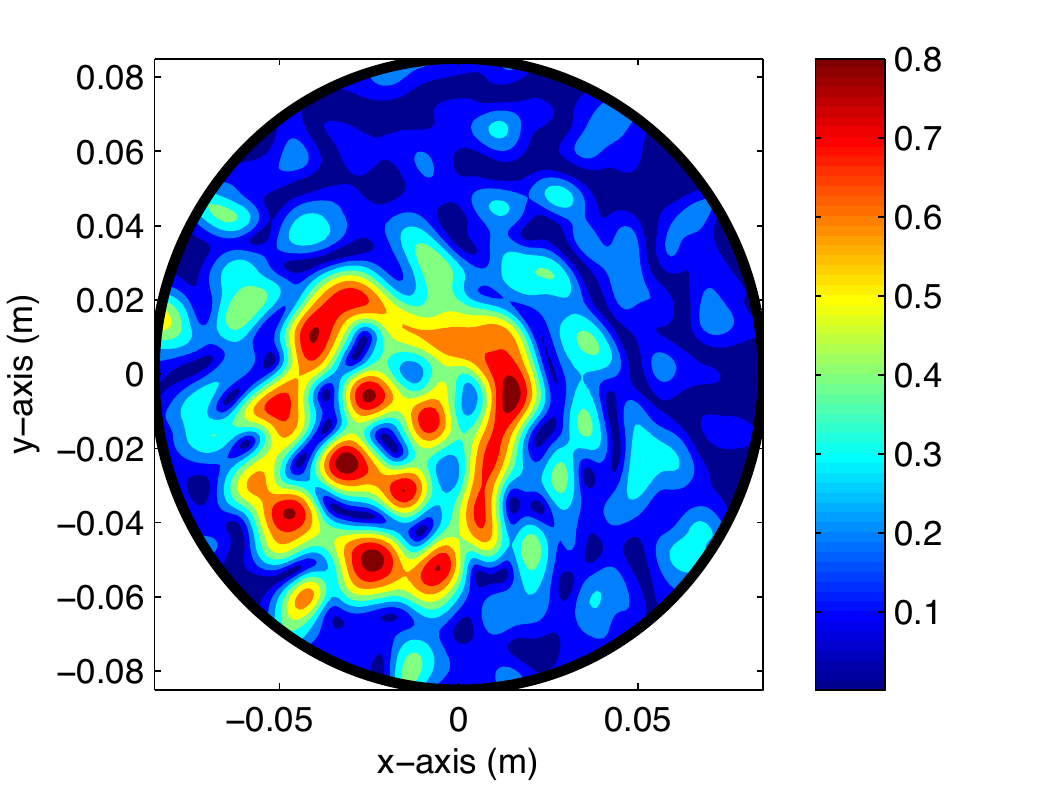}
\caption{\label{Result-Real-Limit}(Section \ref{Section-Real-Limit}) Top line: description of real-data simulations for multiple small anomalies (left) and a rectangular extended anomaly (right). Bottom line: maps of $\mathfrak{F}_{\mathrm{Diag}}(\mr)$ at $f=925~\mathrm{MHz}$ for small anomalies (left) and a rectangular anomaly (right).}
\end{center}
\end{figure}

\section{Conclusion}\label{sec:5}
In this study, we suggested an microwave imaging technique without diagonal elements in the scattering matrix to reflect the effects of anomalies alone. To show the feasibility of the imaging technique, we explored traditional and suggested imaging functions by finding a relationship with infinite series of Bessel function of integer order. Based
 on explored structures, we confirmed that the shapes of small anomalies can be retrieved by imaging function without diagonal elements of the scattering matrix, and this is a considerable improvement over the traditional technique.

Following several works \cite{AGKPS,P-SUB3,P-SUB5,PP2}, it has been shown that subspace migration is effective in a limited-aperture inverse scattering problem. The application and development of a microwave imaging technique to the limited-aperture problem is an interesting problem. In this paper, we considered the imaging of circle-like small and extended targets. Application to the imaging of arbitrarily shaped targets, such as cracks, star-shaped domains, etc., is a subject for future research. Furthermore, there exists some limitations on the real-world applications. We believe that application of multiple frequencies will guarantee better results. Improvement of the designed algorithm for obtaining better results in real-world applications will be the subject of the forthcoming work.

\section*{Acknowledgement}
The author wish to thank professor Jin Keun Seo for his valuable advices and encouragement. The author is also grateful to Kwang-Jae Lee and Seong-Ho Son at the Radio Technology Research Department, Electronics and Telecommunications Research Institute (ETRI) for helping in generating scattering parameter data from CST STUDIO SUITE and microwave machine. The constructive comments of two anonymous reviewers are acknowledged. This research was supported by the Basic Science Research Program through the National Research Foundation of Korea (NRF) funded by the Ministry of Education (No. NRF-2017R1D1A1A09000547).

\bibliographystyle{elsarticle-num-names}
\bibliography{../../References}
\end{document}